\documentclass[11pt, reqno, letterpaper]{amsart}

\usepackage{amssymb}

\usepackage{color}

\usepackage{bbm}

\usepackage[margin = 1in]{geometry}

\usepackage{hyperref}
\usepackage{url}

\allowdisplaybreaks   

\theoremstyle{plain}
\newtheorem{de}{Definition}[section]
\newtheorem{lem}[de]{Lemma}
\newtheorem{prop}[de]{Proposition}
\newtheorem{cor}[de]{Corollary}
\newtheorem{thm}[de]{Theorem}
\newtheorem{lemma}[de]{Lemma}

\theoremstyle{definition}
\newtheorem{rem}[de]{Remark}
\newtheorem{exa}[de]{Example}

\newcommand{\supp}{\operatorname{supp}}

\DeclareMathOperator{\curl}{curl}
\DeclareMathOperator{\Div}{div}
\DeclareMathOperator{\tr}{tr}

\newcommand{\RR}{{\mathbb{R}}}
\newcommand{\NN}{{\mathbb{N}}}

\newcommand{\cB}{{\mathcal{B}}}

\newcommand{\cH}{{\mathcal{H}}}

\newcommand{\ph}{\varphi}
\newcommand{\ep}{\varepsilon}

\newcommand{\D}{\mathrm{d}}
\newcommand{\dd}{\,\mathrm{d}}
\newcommand{\e}{\mathrm{e}}

\newcommand{\sym}{\mathrm{sym}}
\newcommand{\maxi}{\mathrm{max}}
\newcommand{\tdiv}{\mathrm{div}}
\newcommand{\tcurl}{\mathrm{curl}}
\newcommand{\lin}{\mathrm{lin}}
\newcommand{\nl}{\mathrm{nl}}

\newcommand{\DIV}{\operatorname{div}}
\newcommand{\CURL}{\operatorname{curl}}

\newcommand{\hra}{\hookrightarrow}

\newcommand{\beq}{\begin{equation}}
\newcommand{\eeq}{\end{equation}}

\newcommand{\ol}{\overline}
\newcommand{\wh}{\widehat}

\numberwithin{equation}{section}

\newcommand{\vertiii}[1]{{\left\vert\kern-0.25ex\left\vert\kern-0.25ex\left\vert #1 
    \right\vert\kern-0.25ex\right\vert\kern-0.25ex\right\vert}}

\begin{document}


\title{Exponential decay of quasilinear Maxwell equations with interior conductivity}

\author{Irena Lasiecka}
\address{I. Lasiecka, Department of Mathematical Sciences,
University of Memphis, Memphis, TN 38152, USA}
\email{lasiecka@memphis.edu}

\author{Michael Pokojovy}
\address{M. Pokojovy, Department of Mathematical Sciences, University of Texas at El Paso,
500 W University Ave, El Paso, TX 79968, USA}
\email{mpokojovy@utep.edu}

\author{Roland Schnaubelt}
\address{R. Schnaubelt, Department of  Mathematics,
Karlsruhe Institute of Technology, 76128 Karlsruhe, Germany}
\email{schnaubelt@kit.edu}

\thanks{MP and RS gratefully acknowledge financial support from the Deutsche Forschungsgemeinschaft (DFG) through CRC 1173.
IL research is partially supported by the NSF-DMS Grant Nr  1713506.}



\date\today

\begin{abstract}
We consider a quasilinear nonhomogeneous, anisotropic Maxwell system in a bounded smooth domain of $\mathbb{R}^{3}$
with a strictly positive conductivity subject to the boundary conditions of a perfect conductor.
Under appropriate regularity conditions, adopting a classical $L^{2}$-Sobolev solution framework,
a nonlinear energy barrier estimate is established for local-in-time $H^{3}$-solutions to the Maxwell system
by a proper combination of higher-order energy and observability-type estimates under a smallness assumption 
on the initial data.  
Technical complications due to quasilinearity, anisotropy and the lack of solenoidality, etc., are addressed.
Finally, provided the initial data are small, the barrier method is applied to prove that local solutions exist globally 
and exhibit an exponential decay rate.
\end{abstract}

\maketitle



\section{Introduction}
\label{SECTION_INTRODUCTION}

In this work we establish global existence and exponential decay for the quasilinear Maxwell system with strictly positive 
conductivity and small initial fields.
Being the foundation of electro-magnetic theory, the Maxwell equations 
\begin{equation*}
	\partial_{t} D = \CURL H - J \quad \text{ and } \quad
	\partial_{t} B = -\CURL H, \qquad t > 0, \ \  x \in \Omega,
	\label{EQUATION_FARADAY_AND_AMPERE_LAWS}
\end{equation*}
connect the electric fields $E$ and $D$, the magnetic
fields $B$ and $H$, and the current $J$ via Amp\`{e}re's circuital law and  Faraday's law of induction.
Here $\Omega\subseteq \RR^3$ is a simply connected, bounded domain with a smooth boundary $\Gamma$
in $C^5$ and outer unit normal $\nu$.
In our analysis, we take $(E,H)$ as the state variables and postulate the instantaneous nonlinear material laws
\begin{equation*}
    \label{EQUATION_NONLINEAR_INSTANTANEOUS_MATERIAL_LAWS}
    D = \varepsilon(x, E) E \quad \text{ and } \quad B = \mu(x, H) H
\end{equation*}
with nonlinear, nonhomogeneous, anisotropic tensor-valued permittivity $\ep$ and permeability $\mu$.
 We further employ linear Ohm's law
\begin{equation*}
    \label{EQUATION_OHM_LAW_GENERAL}
    J = \sigma(x) E
\end{equation*}
with a nonhomogeneous, anisotropic conductivity tensor $\sigma$. Imposing the boundary conditions
of a perfect conductor, we arrive at the quasilinear Maxwell system
\begin{align}
    \notag
    \partial_t \big(\ep(x,E(t,x))E(t,x)\big) &= \curl H(t,x) -\sigma(x) E(t,x), & t &\ge 0, \ x\in \Omega, \\
    \notag
    \partial_t \big(\mu(x,H(t,x))H(t,x)\big) &= -\curl E(t,x), & t &\ge 0, \ x\in \Omega, \\
    \label{eq:maxwell0}
    E(t,x) \times \nu(x)&=0, \quad \nu(x) \cdot \mu\big(x,H(t,x)\big)H(t,x)=0, & t &\ge0, \ x\in \Gamma, \\
    \notag
    E(0, x)&=E_0(x),  \qquad H(0, x)=H_0(x), & x &\in \Omega,
\end{align}
with initial fields $(E_0, H_0)$ that satisfy the compatibility conditions \eqref{ass:comp}.
We note that the Gaussian laws  \eqref{eq:div0} for charges and the magnetic boundary condition in \eqref{eq:maxwell0}
follow from \eqref{ass:comp} and the other equations in \eqref{eq:maxwell0}. 

In \eqref{ass:coeff1} and \eqref{ass:coeff2} to follow we impose symmetry assumptions on $\ep$ and $\mu$ under which
\eqref{eq:maxwell0} becomes a symmetric quasilinear hyperbolic system. For such systems in the full space 
case $\Omega=\RR^3$, one has a well-developed local well-posedness theory of $\cH^3$--valued solutions due to Kato \cite{Ka}.
However, the above problem has a characteristic boundary which could lead to a loss of regularity in the normal direction.
The available general existence results work in Sobolev spaces of very high order and with weights vanishing at the 
boundary, see \cite{Gu, Se1996} 
as well as \cite{Ra1985} for tangential regularity.
For absorbing boundary conditions, local existence results in $\cH^3$ were given in  \cite{PiZa1995}. However, in our case a 
local well-posedness theory in $\cH^3$ was established only recently in the papers \cite{Sp0, Sp2, Sp1}.
We will strongly rely on these results.

In our main Theorem~\ref{thm:main} we show that local solutions are indeed global and exhibit exponential decay rates
in $\cH^3$, provided that the initial fields are small and that the conductivity is strictly positive.
In  Remark~2 of \cite{El},  it is explained that certain solutions do not decay to 0 
if one drops the simple connectedness of  $\Omega$ or the magnetic compatibility conditions in \eqref{ass:comp},
even for linear $\ep$ and $\mu$.
It should be emphasized that, while decay rates for the Maxwell system have been studied in a number of works 
(viz.\ \cite{AnPo2018, El, ElLaNi2002.2, ET12, komornik, lagnese, NiPi2004, phung} and references therein), 
the cited studies only allow for linear permittivity and permeability and partly deal with constant isotropic coefficients.
Stabilization results for general hyperbolic systems typically concern damping mechanism acting on all components
of the solutions, see \cite{rauch},  whereas in our
Maxwell system the dissipation via conductivity only affects the electric field. For $\Omega=\RR^m$
the paper  \cite{BZ} allows for partial damping even in the quasilinear case, but its assumptions exclude 
the Maxwell equations. For the quasilinear Maxwell system we are only aware of a few
results on the full space  $\Omega=\RR^3$ that establish global existence and decay for small and
smooth solutions, see \cite{LZ, Ra, Spe}. These works rely on dispersive estimates which are not available on 
bounded domains. On the other hand, it is known that blow-up in $W^{1,\infty}$ or $\cH(\curl)$ can occur 
in various cases, see \cite{DNS} and the references therein. 
Up to date, no global results on \emph{quasilinear} Maxwell systems with boundary conditions are known 
in the literature. 

In the linear case, the above mentioned  decay results are based on energy estimates exhibiting dissipation and 
observability-type lower estimates for this dissipation. See \cite{phung} for constant scalar $\ep,\mu>0$
and the very recent
contribution \cite{El} for the general case. For the quasilinear problem, in Propositions~\ref{prop:energy} and \ref{prop:lower}
we show such inequalities for the fields $(E,H)$ and their time derivatives, where we have to admit
corrections terms involving also  space derivatives of the fields.  As in \cite{El, phung}, we use 
vector potentials and Helmholtz decompositions of the fields in the proof of 
the observability-type estimate in Proposition~\ref{prop:lower}. The correction terms are small up to a certain time
if the data are small, but they are small in much stronger norms than the quantities controlled by dissipation.
Therefore, to deal with the quasilinear situation one has to establish an additional, rather deep
regularity result with constants independent  of the time interval. It is provided by our Proposition~\ref{prop:ze}, 
where we bound the spatial derivatives up to order $3-k$ of the fields $\partial_t^k(E,H)$ by the $L^2$ norms of the 
time derivatives alone, where $k\in\{0,1,2\}$.
With this result at hand, one can then  easily show Theorem~\ref{thm:main}
by adopting the widely known ``barrier method'' from the linear case and using smallness, see Section~\ref{SECTION_HIGHER_ENERGY_OBSERVABILITY}. 
In a different situation, a similar procedure was used in \cite{MePoSH2007, MRRa2002} and recently in \cite{LPW} for quasilinear thermoelastic plate equations.

The lengthy proof of Proposition~\ref{prop:ze} is relegated to the last section. Its very first step still
is rather simple. Using the Maxwell system \eqref{eq:maxwell0} one can estimate the $L^2$ norm of 
$\curl H(t)$ by that of $\partial_t E(t)$. (We note that we work on a time interval where we can control, e.g., the $\cH^3$ norm of the fields uniformly.)
We further have the magnetic boundary condition in \eqref{eq:maxwell0}
and the divergence relation $\Div(\mu(H)H)=0$ from \eqref{eq:div0}. A variant of well-known ``elliptic''
$\tcurl$-$\tdiv$ estimates then yields the bound $\|H(t)\|_{\cH^1}\le c\,\|\partial_t E(t)\|_{L^2}$.
This procedure also works for time and tangential space derivatives of $H$,
 but not for normal ones since they destroy the boundary conditions. For the electric field
this approach  entirely fails because the divergence relation \eqref{eq:div0} for $E$ is spoiled by the conductivity term. 
For $E$ and its tangential derivatives  we have to resort to the (weaker) energy estimate taking advantage 
of the better properties of $H$. The normal derivatives of both fields  are then treated by a ``$\tcurl$-$\tdiv$-strategy:''
Using formulas derived in Section~\ref{SECTION_AUXILIARY_RESULTS} one can solve in the Maxwell system
for the normal derivative of the tangential components of the fields, and in the divergence relations for
 the normal derivative of the normal components. In the latter case the anisotropy of the material laws becomes
 a major problem. In the proof of Proposition~\ref{prop:ze} one has to apply these ideas on differentiated
 modifications of the Maxwell system and perform an intricate iteration over the regularity levels.

We briefly outline the rest of the paper. In the next section, our functional-analytic setting 
along with basic notations is introduced and the main result is stated.
In Section~\ref{SECTION_ENERGY_AND_OBSERVABILITY_INEQUALITIES}, we establish energy and observability-type inequalities for 
local solutions to system (\ref{eq:maxwell0}). Subsequently, in Section~\ref{SECTION_HIGHER_ENERGY_OBSERVABILITY}, these 
estimates are improved to incorporate higher-order spatial derivatives which then allows us to show the main result.
Section~\ref{SECTION_AUXILIARY_RESULTS} presents elliptic-type $\tcurl$-$\tdiv$-estimates and introduces  our 
$\tcurl$-$\tdiv$-strategy, which are subsequently adopted in Section~\ref{SECTION_REGULARITY_RESULT} to prove
our core regularity result, i.e., Proposition~\ref{prop:ze}.

\section{Problem setting and main result}
\label{SECTION_SETTINGS_AND_RESULTS}

Let $\Omega\subset \RR^3$ be a bounded domain with a boundary $\Gamma := \partial \Omega$ of class $C^5$ 
and the outer unit normal vector $\nu$.
For $T>0$, we set $J = J_T =[0, T]$, $\Omega_T = (0,T)\times \Omega$ and $\Gamma_T= (0,T)\times \Gamma$.
For the sake of brevity, the same notation will often be used for spaces of scalar and vector-valued functions.
Also, we sometimes write $\cH^k$ instead of the Sobolev space $\cH^{k}(\Omega)$, etc., if the domain of integration 
is clear from the context.  Spaces on $\Gamma$ are always equipped with the surface measure denoted by $\dd x$.  
Our basic assumptions  are 
\begin{equation}
\label{ass:coeff1}
\begin{split}
    &\ep,\mu \in C^3(\ol{\Omega}\times \RR^3, \RR^{3\times3}_{\sym}), \quad
    \sigma\in  C^3(\ol{\Omega}, \RR^{3\times3}_{\sym}) \quad \text{ and } \\
    &\ep(x,0) \ge 2\eta I, \quad \mu(x,0) \ge 2\eta I, \quad \sigma(x)\ge \eta I \qquad \text{for all } x\in \ol{\Omega}
\end{split}
\end{equation}
and for some constant $\eta > 0$, where $C^3(\ol{\Omega})$ is the space of $C^3$--functions $v$
such that $v$ and its derivatives up to the third order possess a continuous extension to $\Gamma$. 
We introduce the matrix-valued functions $\ep^\D$ and $\mu^\D$ given by
\begin{equation*}
	\ep^\D_{jk}(x,\xi)= \ep_{jk}(x,\xi) + \sum_{l=1}^ 3\partial_{\xi_k}\ep_{jl}(x,\xi) \,\xi_l , \qquad
    \mu^\D_{jk}(x,\xi)= \mu_{jk}(x,\xi) + \sum_{l=1}^ 3\partial_{\xi_k}\mu_{jl}(x,\xi) \,\xi_l
\end{equation*}
for $x\in \ol{\Omega}$, $\xi\in\RR^3$ and $j,k\in\{1,2,3\}$, which arise when differentiating the  left-hand 
side of   \eqref{eq:maxwell0}. We further assume that
\begin{equation}
    \label{ass:coeff2}
    \partial_\xi \ep,\partial_\xi \mu \in C^3(\ol{\Omega}\times \RR^3, \RR^{3\times3}), \quad
    \ep^\D=(\ep^\D)^\top, \quad \text{ and } \quad \mu^\D=(\mu^\D)^\top.
\end{equation}
By continuity, there exists a radius $\tilde{\delta} \in (0,1]$ such that
\begin{equation}
    \label{est:coeff-lower}
    \ep(x,\xi), \mu(x,\xi), \ep^\D(x,\xi), \mu^\D(x,\xi), \sigma(x)\ge \eta I
\end{equation}
for all $\xi \in\RR^3$ with $|\xi|\le \tilde{\delta}$ and $x\in \ol{\Omega}$.

\begin{exa}
Let $\ep_{\lin}\in C^3(\ol{\Omega}, \RR^{3\times3}_{\sym})$ satisfy $ \ep_{\lin}\ge 2\eta I$. We specify two nonlinear 
terms in the sum $\ep(x,E) = \ep_{\lin}(x) + \ep_{\nl}(x,E)$ so that
the conditions  \eqref{ass:coeff1} and \eqref{ass:coeff2}  are valid.  One can take Kerr-type isotropic nonlinearities 
$\ep_{\nl}(x,E)= a(x)\ph(|E|^2)I$ for scalar functions $a\in  C^3(\ol{\Omega})$ and $\ph\in C^4([0, \infty))$ 
with $\ph(0)=0$. A typical anistropic  example is furnished by
\[ \ep_{\nl}(x,E) = \Big( \sum\nolimits_{j,k=1}^3 \chi_i^{jkl}(x)E_jE_k\Big)_{il}\]
for scalar coefficients $ \chi_i^{jkl}\in C^3(\ol{\Omega})$, cf.\ \cite{MN}. Because of the triple sum
in $\ep_{\nl}(x,E)E$,  the tensor  $(\chi_i^{jkl})_{i,j,k,l}$  has to be 
symmetric in $\{j,k,l\}$. Our assumptions also require symmetry in $\{i,l\}$, i.e., we can only
prescribe $\chi_i^{jkl}$ for, say, $1\le i\le j\le l\le 3$.
\end{exa}

We write $\tr_n u$ for the trace of the normal component $u\cdot \nu$ on $\Gamma$,
while $\tr_t u$ stands for the tangential trace $u\times \nu$ on $\Gamma$. 
We also use its rotated counterpart
$\tr_\tau u = \nu \times (\tr_t u)$, which is the tangential component $\tr u-(\tr_n u) \nu$ of the full trace $\tr u$. 
It is well known that the mappings 
\begin{equation*}
    \tr_n : \cH(\tdiv)\to \cH^ {-1/2}(\Gamma) \quad \text{ and } \quad \tr_t : \cH(\tcurl)\to \cH^ {-1/2}(\Gamma)^3
\end{equation*}
are continuous, where the Hilbert spaces
\begin{align*} 
    \cH(\tcurl) &= \big\{u\in L^2(\Omega)^3\,|\, \curl u \in L^2(\Omega)^3\big\}  \quad \text{ and } \quad
    \cH(\tdiv) = \big\{u\in L^2(\Omega)^3\,|\, \Div u \in L^2(\Omega)\big\}
\end{align*} 
are endowed with their natural norms. See Theorems~IX.1.1 and IX.1.2 in \cite{DL}.
 
Let $E_0, H_0\in \cH^3(\Omega)^3$. To express the compatibility conditions,
we set
\begin{align}\label{def:comp}
    E_0^1&= \ep^\D(E_0)^{-1}[\curl H_0-\sigma E_0], \qquad H_0^1= -\mu^\D(H_0)^{-1}\curl E_0,\notag\\
    E_0^2&=   \ep^\D(E_0)^{-1}\big[\curl H_0^1-\sigma E_0^1 - (\nabla_E\ep^\D(E_0) E_0^1)\cdot E_0^1\big],\\
    H_0^2&=  - \mu^\D(H_0)^{-1}\big[\curl E_0^1 + (\nabla_H\mu^ \D(H_0)H_0^ 1)\cdot H_0^1\big],\notag
\end{align}
where we put $((\nabla A)\xi\cdot \eta)_j= \big(\sum_{i,k}\partial_i A_{jk} \xi_k \eta_i\big)_j$.
The initial fields shall satisfy the divergence and boundary conditions
\begin{equation}
    \label{ass:comp}
        \Div(\mu(H_0)H_0)=0, \quad  \tr_n(\mu(H_0)H_0)=0, \qquad
        \tr_t E_0 = \tr_t E_0^1= \tr_t E_0^2=0.
\end{equation}
Letting $C_S>0$ be the norm of the Sobolev embedding $\cH^2(\Omega)\hra L^ \infty(\Omega)$,
we put $\delta_0 = \min\{1,\tilde{\delta}/C_S\}$. 

Take $T > 0$ and $\delta\in(0,\delta_0]$. The (small) parameter $\delta>0$ will be fixed in subsequent proofs. 
The local well-posedness result Theorem~5.3 in \cite{Sp2}
provides a radius $r(T,\delta)\in  \big(0, r(T,\delta_0)\big]$ such that, for all $r\in \big(0, r(T,\delta)\big]$
and initial data  $E_0, H_0 \in \cH^3(\Omega)^3$ fulfilling the conditions \eqref{ass:comp} and
 the smallness assumption
\begin{equation}
    \label{est:data}
    \|E_0\|_{\cH^3(\Omega)}^2 + \|H_0\|_{\cH^3(\Omega)}^2 \le r^2,
\end{equation}
there exists a maximal existence time $T_{\maxi}\in (T,\infty]$ and a unique solution
\begin{equation}\label{local}
    (E,H)\in \bigcap_{k=0}^3 C^k\big([0, T_{\maxi}), \cH^{3-k}(\Omega)\big)^6=:G^3
\end{equation}
to the quasilinear Maxwell system \eqref{eq:maxwell0}. 
The fields $(E, H)$ further satisfy the estimate
\begin{equation}
    \label{est:delta}
    \max_{k\in\{0,1,2,3\}} \max_{t \in [0, T]} \big(\|\partial_t^k E(t)\|_{\cH^{3-k}(\Omega)}^2 
        + \|\partial_t^k H(t)\|_{\cH^{3-k}(\Omega)}^ 2 \big)\le \delta^2\le1
\end{equation}
and the divergence equations
\begin{equation}
    \label{eq:div0}
    \begin{split}
    \Div\big(\mu(H(t))H(t)\big) &= 0, \\
    \Div\big(\ep(E(t))E(t)\big) &= \Div\big(\ep(E_0)E_0\big) - \int_0^t \Div \big(\sigma E(s)\big) \dd s.
    \end{split}
\end{equation}
on $\Omega$ for all $t\in [0,T_{\maxi})=:J_{\maxi}$. (We  write $E(t)$ instead of $E(t,\cdot)$, etc.)

We note that Theorem~5.3 in \cite{Sp2}  is not concerned with \eqref{eq:div0} and  the second boundary condition in the system
\eqref{eq:maxwell0}. These formulas follow from the other equations in
 \eqref{eq:maxwell0} and the assumption \eqref{ass:comp} in a standard way, see Lemma~7.25 of \cite{Sp0}.

Inequality \eqref{est:delta} will frequently be invoked in this article, sometimes without being explicitly mentioned. 
In addition to rendering the solution small, it also provides a crucial uniform bound. 
Observe that along solutions to \eqref{eq:maxwell0} fulfilling \eqref{est:delta}, 
the lower bound \eqref{est:coeff-lower} is valid for $t \in [0, T]$.

We now fix $T = 1$ yielding the radius $r(\delta):= r(\delta,1)$. 
Given initial fields $(E_0, H_0)$ satisfying  \eqref{ass:comp} and \eqref{est:data}, we introduce the time
\begin{equation}
    \label{def:T*}
    T_* = \sup\big\{ T\in[1,T_{\maxi})\,|\, \eqref{est:delta} \text{ is valid for } t\in[0,T]\big\}.
\end{equation}
The bound \eqref{est:delta} is thus true on $[0,T_*)=:J_*$. If $T_*<\infty$, then the blow-up condition in Theorem~5.3
of \cite{Sp2} implies that $T_{\maxi} >T_*$ and hence
\beq \label{eq:contr}
z(T_*):=\max_{k\in\{0,1,2,3\}} \big(\|\partial_t^k E(T_*)\|_{\cH^{3-k}(\Omega)}^2 
        + \|\partial_t^k H(T_*)\|_{\cH^{3-k}(\Omega)}^ 2 \big) = \delta^2 \qquad (\text{if \ }T_*<\infty)
\eeq
by continuity. 

We work with time-differentiated versions of \eqref{eq:maxwell0}.  For the sake of brevity, we set
\begin{equation}
    \label{def:ep-mu-hat}
    \wh \ep_k = \begin{cases} \ep(E),& k=0, \\  \ep^\D(E),& k\in\{1,2,3\},\end{cases} \qquad
    \wh \mu_k = \begin{cases} \mu(H),& k=0, \\ \mu^\D(H),& k\in\{1,2,3\}.\end{cases}
\end{equation}   
For $k \in \{0,1,2,3\}$, we then obtain the system
\begin{align} 
    \label{eq:maxwell1}
    \partial_t (\wh\ep_k \partial_t^kE)&= \curl \partial_t^k H -\sigma \partial_t^k E - \partial_t f_k, 
    & t &\in J_{\maxi}, \ x\in \Omega, \notag \\
    \partial_t (\wh\mu_k\partial_t^k H) &= -\curl \partial_t^k E -\partial_t g_k, & t &\in J_{\maxi}, \ x\in \Omega,\\
    \tr_t\partial_t^ k E &=0, \quad \tr_n(\wh\mu_k\partial_t^k H) = -\tr_n g_k, & t &\in J_{\maxi},\ x\in\Gamma,\notag
\end{align}
with the commutator terms
\begin{equation}
    \label{def:fk}
    \begin{split}
    f_0&=f_1=0, \ \ f_2= \partial_t \ep^\D(E) \,\partial_t E, \ \
    f_3= \partial_t^2 \ep^\D(E) \, \partial_t E + 2  \partial_t\ep^ \D(E)\, \partial_t^2 E,\\
    g_0&=g_1=0, \ \ g_2= \partial_t \mu^\D(H) \,\partial_t H, \ \
    g_3= \partial_t^2 \mu^\D(H) \, \partial_t H + 2  \partial_t\mu^\D(H)\, \partial_t^2 H.
    \end{split}
\end{equation} 
Equation \eqref{eq:div0} further yields the divergence relations
\begin{equation}
    \label{eq:div}
    \Div(\mu^\D(H)\partial_t^ kH)= - \Div g_k,\qquad
    \Div(\ep^\D(E)\partial_t^ kE)= -\Div(\sigma \partial_t^{k-1} E +f_k)
\end{equation} 
for $k\in\{1,2,3\}$. Estimate \eqref{est:z} below shows that all functions $\partial_t f_k$, $\partial_t g_k$,
$\Div f_k$, and $\Div g_k$ belong to $L^\infty(J_*, L^2(\Omega))$.
For $k = 3$, the evolution equations in \eqref{eq:maxwell1} are interpreted in $\cH^{-1}(\Omega_T)$
while the divergence operator in \eqref{eq:div} is understood in $\cH^{-1}(\Omega)$. 
Since the inhomogenities belong to $L^2$, the traces in \eqref{eq:maxwell1}  exist 
in $\cH^ {-1/2}(\Gamma)$, cf.\ Section~2.1 of \cite{Sp0}.

For the energy estimate, it is useful to consider an equivalent version of \eqref{eq:maxwell1}, viz.\
\begin{align} 
    \label{eq:maxwell2}
    \ep^ \D(E)\, \partial_t\partial_t^kE&= \curl \partial_t^k H -\sigma \partial_t^k E -  \tilde{f}_k,
    & t&\in J_{\maxi}, \ x\in \Omega, \notag \\
    \mu^ \D(H)\,\partial_t\partial_t^k H&= -\curl \partial_t^k E -\tilde{g}_k, & t &\in J_{\maxi}, \ x\in \Omega,\\
    \tr_t \partial_t^ k E &=0, & t &\in J_{\maxi}, \ x\in \Gamma, \notag
\end{align}
for $k\in\{0,1,2,3\}$ with the new commutator terms
\begin{equation} 
    \label{eq:tilde-fk}
    \tilde{f}_k = \sum_{j=1}^k \binom{k}{j} \partial_t^ j \ep^ \D(E)\, \partial_t^{k+1-j} E,\qquad
    \tilde{g}_k = \sum_{j=1}^k \binom{k}{j} \partial_t^ j \mu^ \D(H)\, \partial_t^{k+1-j} H,
\end{equation} 
where we put $\tilde{f}_0 = \tilde{g}_0 = 0$. We further introduce the quantities
\begin{align} 
    \label{def:dez}
    e_k(t)&= \frac12\max_{j\in \NN_0, j\le k}\big( \|\wh\ep_k^{1/2} \partial_t^j E(t)\|_{L^2(\Omega)}^2 
      + \|\wh\mu_k^{1/2} \partial_t^j H(t)\|_{L^2(\Omega)}^2 \big),  \qquad e=e_3,\notag\\
    d_k(t)&= \max_{j\in \NN_0, j\le k} \|\sigma^ {1/2} \partial_t^j E(t)\|_{L^2(\Omega)}^2,\qquad d=d_3, \\
    z_k(t)&= \max_{j\in \NN_0, j\le k} \big( \|\partial_t^j E(t)\|_{\cH^{k-j}(\Omega)}^2 
      + \|\partial_t^j H(t)\|_{\cH^{k-j}(\Omega)}^2 \big), \qquad z=z_3, \notag
\end{align}
for $k \in \{0,1,2,3\}$ and $t\in J_{\maxi}$. The coefficients of the energies  $e_k$ are chosen in view of Lemma~\ref{lem:helmholtz}.
Throughout the paper, $c_k$ or $c$ are positive constants that do not depend on $t \in [0, T_*)$, $T_*$, $\delta\in(0,\delta_0]$, $r\in (0,r(\delta_0)]$, 
and $(E_0,H_0)$ satisfying the conditions \eqref{ass:comp} and \eqref{est:data}.

Using standard  methods (as in  Section~2 of \cite{Sp2}) and the estimate \eqref{est:delta}, one can show that
\begin{equation} 
    \label{est:z}
    \begin{split}
    \|\wh\ep_k(t)\|_\infty, \|\wh\mu_k(t)\|_\infty, \|\wh\ep_k^{-1}(t)\|_\infty, \|\wh\mu_k^ {-1}(t)\|_\infty&\le c,\\
    \|\partial^ \alpha\wh\ep_j(t)\|_{L^2(\Omega)}, \|\partial^ \alpha\wh\mu_j(t)\|_{L^2(\Omega)} 
       &\le c(z_k^{1/2}(t)+\delta_{\alpha_0=0}),\\
    \max_{k\in\{2,3\},j\in\{0,1\}}\big( \|\partial_t^ j f_k(t)\|_{\cH^{4-j-k}(\Omega)} 
    + \|\partial_t^ j g_k(t)\|_{\cH^{4-j-k}(\Omega)}\big)   &\le cz(t),\\
    \|f_2(t)\|_{L^2(\Omega)},\|g_2(t)\|_{L^2(\Omega)},\|f_3(t)\|_{L^2(\Omega)},\|g_3(t)\|_{L^2(\Omega)}
     &\le ce_2^{1/2}(t),\\
    \|\tilde f_k(t)\|_{\cH^{3-k}(\Omega)}, \|\tilde g_k(t)\|_{\cH^{3-k}(\Omega)} &\le cz(t)
    \end{split} 
\end{equation} 
for $j , k\in \{0,1,2,3\}$, $\alpha\in\NN_0^4$ with $|\alpha|=k>0$, and $ t\in J_*$.
The constants $c$ do not depend on $t$, and we set $\partial_0=\partial_t$, $\delta_{\alpha_0=0}=1$ if $\alpha_0=0$, 
and  $\delta_{\alpha_0=0}=0$ if $\alpha_0>0$. The term $+c$ on the right-hand side of the second line in \eqref{est:z}
arises if all derivatives in $\partial^\alpha$ are applied to the $x$--variable of $\ep$ or $\mu$.

The main goal of this paper is to establish the global existence of the local solutions in \eqref{local}, assuming 
that the initial data are small enough. It is well known that global existence for quasilinear systems is closely related to 
the exponential decay of the resulting dynamics. The main bulk of the paper is thus devoted to the proof of this latter property. 
Our main result reads as follows.

\begin{thm}
    \label{thm:main}
    Let $\Omega \subset \RR^3$ be a bounded, simply connected domain with $\partial\Omega\in C^5$, 
    the coefficients satisfy  \eqref{ass:coeff1} and \eqref{ass:coeff2}, and the initial data $E_0, H_0\in \cH^3(\Omega)^3$ 
    fulfill \eqref{ass:comp} and \eqref{est:data}. Then there exists a radius $r>0$ in 
    assumption \eqref{est:data} and constants 
    $M, \omega>0$ such that the solution $(E, H)$ of the Maxwell system \eqref{eq:maxwell0} exists  for all $t \ge 0$
    and is bounded by
    \begin{equation*}
        \max_{j\in \NN_0, 0\le j\le 3} \big( \|\partial_t^j E(t)\|_{\cH^{3-j}(\Omega)}^2 
        + \|\partial_t^j H(t)\|_{\cH^{3-j}(\Omega)}^2 \big) 
          \le M\e^{- \omega t} \|(E_0, H_0)\|_{\cH^3(\Omega)}^2 \quad \text{ for all } t \ge 0.
    \end{equation*}
\end{thm}
The proof of the theorem is given at the end of Section~\ref{SECTION_HIGHER_ENERGY_OBSERVABILITY}.

\section{Energy and observability-type inequalities}
\label{SECTION_ENERGY_AND_OBSERVABILITY_INEQUALITIES}

We start with a basic higher-order energy estimate establishing an explicit dissipation in the system due to the electric conductivity. 
\begin{prop}
    \label{prop:energy}
    We assume the conditions of Theorem~\ref{thm:main} except for the simple connectedness of $\Omega$.
    For $0\le s\le t< T_*$ and  $k\in \{0,1,2,3\}$, we obtain the inequality
    \begin{equation}\label{est:energy}
        e_k(t)+ \int_s^ t d_k(\tau)\dd \tau \le e_k(s) + c_1  \int_s^ t z^{3/2}(\tau) \dd \tau,
    \end{equation}
    where the constant $c_1$ does not depend on  $s$ and $t$.
\end{prop}
We first give the short proof for the case $k=0$. Since our 
solutions $(E,H)$ of \eqref{eq:maxwell0} are regular, see \eqref{local}, integration by parts and \eqref{eq:maxwell0} easily yield
\begin{align*}
 \frac{\D}{\D t}\,\frac12& \!\int_\Omega \big(\ep(E(t))E(t)\cdot E(t)+ \mu(H(t)) H(t)\cdot H(t)\big) \dd x \\
 &=\frac12\int_\Omega \big(  \partial_t(\ep(E)E) \cdot E + \ep(E) E \cdot (\ep(E)^{-1} \partial_t(\ep(E)E)) +\ep(E)  E \cdot (\partial_t \ep(E)^{-1} \, \ep(E)E))\\
    &\qquad +  \partial_t(\mu(H)H) \cdot H + \mu(H) H \cdot (\mu(H)^{-1} \partial_t(\mu(H)H)) + \mu(H) H \cdot (\partial_t \mu(H)^{-1} \, \mu(H)H))\big)\dd x \\
    &= \int_\Omega \big( \curl H \cdot E -\sigma E\cdot E- \curl E \cdot H  -  \tfrac12 \partial_t \ep(E) \,  E\cdot E
      -  \tfrac12 \partial_t \mu(H) \,H \cdot H\big)\dd x\\
    &= -\int_\Omega \big( \sigma E\cdot E+ \tfrac12\partial_t \ep(E) \,  E\cdot E +   \tfrac12\partial_t \mu(H) \,H \cdot H\big)\dd x.
 \end{align*}
We thus obtain the energy inequality
 \beq\label{eq:energy0}
 e_0(t) + \int_s^ t d_0(\tau)\dd \tau
   = e_0(t) -\frac12 \int_{(s,t)\times \Omega} \big(\partial_t \ep(E) \,  E\cdot E +   \partial_t \mu(H) \,H \cdot H\big)\dd (\tau, x).
 \eeq
Combined with estimate \eqref{est:z}, we derive \eqref{est:energy} for the case $k=0$.

For $k\in \{1,2,3\}$ in Proposition~\ref{prop:energy}, we have  different coefficients in the energy $e_k$ defined in \eqref{def:dez}.
In this case, \eqref{est:energy} follows from  Lemma~\ref{lem:energy} below, the system \eqref{eq:maxwell2} and the estimates \eqref{est:z}. 
This lemma provides an energy identity in a more general situation to be encountered later. 

For some $T > 0$, let the coefficients $a, b\in W^{1,\infty}(\Omega_T, \RR^{3+3}_{\sym})$ satisfy $a, b \ge \eta I$.
Take data $\ph, \psi \in L^2(\Omega_T)^3$, $\chi\in L^2(J, \cH^{1/2}(\Gamma))^3$ with $\nu\cdot\chi =0$, and $u_0, v_0 \in L^2(\Omega)^3$.
Theorem~1.4 of \cite{El12} yields a solution 
$(u,v) \in C(J, L^2(\Omega))^6$ with $\tr_t v \in  L^2(J, \cH^{-1/2}(\Gamma))^3$ to the linear  system
\begin{equation} 
    \label{eq:aux}\begin{split}
    a \partial_t u&= \curl v -\sigma u -  \ph, \qquad t\in J, \ x\in \Omega,  \\
    b \partial_t v&= -\curl u - \psi, \qquad t\in J, \ x\in \Omega,\\
    \tr_t u &=\chi,  \qquad t\in J, \ x\in \Gamma. \\
    u(0)&=u_0, \quad v(0)=v_0.
    \end{split}
\end{equation}
(As noted before \eqref{eq:maxwell2}, the tangential trace of $u$ exists in  $L^2(J, \cH^{-1/2}(\Gamma))$.)
 

\begin{lem} \label{lem:energy}
Under the assumptions above, for $0 \le s \le t\le T$ we have
\begin{align}
    \label{eq:energy}
    \begin{split}
    \frac12\int_\Omega \big(&a(t)u(t)\cdot u(t)+ b(t) v(t)\cdot v(t)\big) \D x 
     + \int_s^t\int_\Omega \sigma u \cdot u\dd x\dd\tau \\
    &= \frac12\int_\Omega \big(a(0)u_0\cdot u_0+ b(0) v_0\cdot v_0 \big) \D x
     + \int_s^t \int_\Gamma \chi \cdot \tr_\tau v \dd x\dd\tau \\
    &\quad  + \int_s^t \int_\Omega \big( \tfrac12 \partial_t a\,  u\cdot u +  \tfrac12\partial_tb\, v\cdot v
            + \ph\cdot u + \psi\cdot v  \big) \D x\dd\tau.  
    \end{split}
\end{align}
\end{lem} 

\begin{proof}
For $\cH^1$--solutions $(u,v)$, the  claim easily  follows from the system \eqref{eq:aux} and integration by parts, see step~3) below.
We thus have to regularize the given data and coefficients to obtain  $\cH^1$--solutions for these regularized problems. Afterwards 
one passes to the limit in the resulting variant of \eqref{eq:energy}. In view of the available a priori estimates and regularity results
from \cite{El12}, \cite{Sp0} or  \cite{Sp1}, one has to approximate the data and the coefficients separately. The assertion is closely 
related to \cite{El12}, but not stated there. Since the reasoning is somewhat involved, we give a (partly sketchy) proof.

\smallskip 

1) We approximate the initial data $u_0$ and $v_0$ and the forcing terms $\ph$ and  $\psi$ in $L^2$ by test functions $u_{0,n}$, 
$v_{0,n}$, $\ph_n$ and $\psi_n$, respectively. The boundary inhomogeneity $\chi$ is approximated in  $L^2(J, \cH^{1/2}(\Gamma))$
by mappings  $\chi_n\in \cH^1(J, \cH^{3/2}(\Gamma))$ which vanish at $t=0$. Moreover, we take coefficients 
$a_m, b_m \in C^3(J\times \ol{\Omega},\RR^{3\times3}_{\sym})$ 
which are uniformly positive definite and uniformly bounded in $W^ {1,\infty}(J\times \ol{\Omega},\RR^{3\times3}_{\sym})$, that converge to $a$ and $b$ uniformly,
and whose derivatives tend pointwise a.e.\ to $\nabla_{t,x} a$ and  $\nabla_{t,x} b$, respectively, as $m\to\infty$.

\smallskip 

2) Theorem~1.1 of \cite{Sp1} yields functions $(u_{n,m},v_{n,m})$ in $G^1:=C^1(J,L^2(\Omega))\cap C(J, \cH^1(\Omega))$ which solve the problem 
\eqref{eq:aux} with the coefficients and the data from step~1).  We note that the required compatibility condition $\tr_t u_{0,n} =\chi_n(0)$
is trivially satisfied. The a priori estimates in this theorem  are \emph{not} uniform in $m$ or $n$. However, Corollary~3.12 of \cite{Sp0}
allows us to dominate $(u_{n,m},v_{n,m})$ in $G^1$ by constants depending on the (uniformly bounded) $W^{1,\infty}$--norms of $a_m$ and $b_m$
as well as on the norms of  $u_{0,n}$,  $v_{0,n}$, $\ph_n$ and $\psi_n$ in $\cH^1$ and of  $\chi_n$ in 
$L^2(J, \cH^{3/2}(\Gamma))\cap \cH^1(J, \cH^{1/2}(\Gamma))$. 
(Note that these norms of the data may blow up as $n\to\infty$.)
This corollary actually deals with the localized problem on the half-space $\RR^3_+$, but it can be transfered to our system \eqref{eq:aux}
on $\Omega$ in a standard way, cf.\ Chapter~5 of \cite{Sp0} or Section~2 of \cite{Sp1}. 
Moreover, Theorem~1.4 of \cite{El12} shows a uniform estimate of the norms of the solutions in $C(J,L^2(\Omega))$ and of their
tangential traces in $L^2(J, \cH^{-1/2}(\Gamma))$ by the norms of the data in $L^2$ or the boundary forcing in $L^2(J, \cH^{1/2}(\Gamma))$.

We first keep $n\in \NN$ fixed. The aforementioned results  from  \cite{Sp0} and \cite{El12} 
imply that a subsequence of  $(u_{n,m},v_{n,m})_m$ has a weak-$\ast$ 
accumulation point $(u_{n},v_{n})$ in $W^{1,\infty}(J,L^2(\Omega))\cap L^ \infty(J, \cH^1(\Omega))$, that  $(u_{n,m},v_{n,m})_m$
converges to $(u_n,v_n)$ in $C(J,L^2(\Omega))$ and  that $(\tr_\tau u_{n,m},\tr_\tau v_{n,m})_m$ tends to $(\tr_\tau u_n, \tr_\tau v_n)$
 in  $L^2(J, \cH^{-1/2}(\Gamma))$.
It is then routine to check that the functions $(u_n,v_n)$ solve  \eqref{eq:aux}
with the coefficients $a$ and $b$ and for the data  $u_{0,n}$, $v_{0,n}$, $\ph_n$, $\psi_n$, and $\chi_n$. 

\smallskip

3)   Using the system \eqref{eq:aux} and integrating by parts, we  calculate
  \begin{align}\label{eq:energy1}
        \frac{\D}{\D t} \; &\frac12\!\int_\Omega \big(a(t)u_n(t)\cdot u_n(t)+ b(t) v_n(t)\cdot v_n(t)\big) \dd x \\
        &= \int_\Omega \big( \tfrac12 \partial_t a\,  u_n\cdot u_n + \tfrac12\partial_t b \, v_n\cdot v_n
        +  a \partial_t u_n\cdot u_n + b  \partial_t v_n\cdot v_n\big)\dd x \notag \\
        &= \int_\Omega \big( \tfrac12 \partial_t a\,  u_n\cdot u_n + \tfrac12\partial_t b\,  v_n\cdot v_n
        + (\curl v_n -\sigma u_n+\ph_n )\cdot u_n +(-\curl u_n +\psi_n)\cdot v_n \big)\dd x\notag \\
        &=  \int_\Omega \big( \tfrac12 \partial_t a\,  u_n\cdot u_n  +  \tfrac12\partial_t b\,v_n\cdot v_n
            + \ph_n\cdot u_n + \psi_n\cdot v_n  -\sigma u_n\cdot u_n\big) \dd x +  \int_\Gamma \chi_n \cdot \tr_\tau v_n \dd x. \notag 
    \end{align}
 
4) The estimate in Theorem~1.4 of \cite{El12} indicated above now implies the convergence of $((u_{n},v_{n}))_n$ to $(u,v)$ in $C(J,L^2(\Omega))$
and of  $((\tr_\tau u_{n},\tr_\tau v_{n}))_n$ to $(\tr_\tau u, \tr_\tau v)$ in  $L^2(J, \cH^{-1/2}(\Gamma))$. Here $(u,v)$ is the solution
to \eqref{eq:aux} provided by  Theorem~1.4 of \cite{El12}. After integrating the identity \eqref{eq:energy1} in  time, we can finally pass to the limit 
 $n\to\infty$ obtaining \eqref{eq:energy}.
 \end{proof}

We now assume that $\Omega$ is simply connected in order to derive our observability-type estimate. 
Following \cite{El} or \cite{phung} in the linear autonomous case, we use Helmholtz decompositions of the fields $(E(t), H(t))$ 
and the spaces
\begin{align*}
    \cH(\tcurl \, 0)&=\{u\in L^2(\Omega)^3\,|\, \curl u =0\}, \qquad
    \hspace*{-0.43cm} \cH_0(\tcurl \, 0)=\{u\in \cH(\tcurl \, 0)\,| \tr_t u =0\},\\
    \cH(\tdiv \, 0)&=\{u\in L^2(\Omega)^3\,|\, \Div u =0\},\qquad
    \!\!\! \cH_0(\tdiv \, 0)=\{u\in \cH(\tdiv \, 0)\,| \tr_n u =0\},\\
    \cH^\Gamma(\tdiv \, 0)&=\big\{u\in \cH(\tdiv \, 0)\,\big|\, \textstyle{\int_{\Gamma_j}} \tr_n u\dd x=0 
         \text{ \ \ for all components $\Gamma_j$ of $\Gamma$}\big\},\\
    \cH^1_{t0}(\Omega)&= \{u\in \cH^1(\Omega)^3\,|\, \tr_t u =0\}=\{u\in \cH(\tdiv)\cap \cH(\tcurl)\,|\,\tr_t u=0\},
\end{align*}
The last identity is shown in Theorem~XI.1.3 of \cite{DL}. The first five spaces are endowed with the $L^2$--norm,
while  $\cH^1_{t0}(\Omega)$ and its subspace $\cH(\tdiv \, 0)\cap \cH_0(\tcurl \, 0)$ are equipped with that of  $\cH^1$. 
We next establish the Helmholtz decomposition needed in the sequel. Our result is a variant of 
Proposition~2 in \cite{El}, where the case of time-independent $\ep$ and $\mu$ and less regular solutions
was treated.

\begin{lem}
    \label{lem:helmholtz}
    Let the assumptions of Theorem~\ref{thm:main} be satisfied. We take the fields  $(E,H)$ from \eqref{local} solving
    the Maxwell system \eqref{eq:maxwell0}. Then there exist functions $w$ in 
    $C^3\big(J_{\maxi}, \cH^ 1_{t0}(\Omega)^3\cap \cH^\Gamma(\tdiv \, 0)\big)\cap C^4\big( J_{\maxi}, L^2(\Omega)\big)^3$, 
    $p$ in $C^3(J_{\maxi}, \cH^1_0(\Omega))$ and $h$ in $C^3\big( J_{\maxi},  \cH(\tdiv \, 0)\cap \cH_0(\tcurl \, 0)\big)$ such that
    \begin{align} 
        \partial_t^k E &= -\partial_t^ {k+1} w + \nabla  \partial_t^k p +  \partial_t^k h, \label{eq:helmholtz1}\\
        \wh\mu_k\partial^k_t H &= \curl \partial_t^ k w -g_k \label{eq:helmholtz2}
    \end{align}
    for $k\in\{0,1,2,3\}$, cf.\ \eqref{def:ep-mu-hat} and \eqref{def:fk}, where the sum in \eqref{eq:helmholtz1} is orthogonal in 
    $L^2(\Omega)^3$.
\end{lem}

\begin{proof}
Let $t\in J_{\maxi}$.    Equations \eqref{eq:maxwell0} and \eqref{eq:div0} imply that the function $\mu\big(H(t)\big) H(t)$ is contained in
    $\cH_0(\tdiv\,0)$. Because $\Omega$ is simply connected, Theorem~2.8 of \cite{Ce} then yields a vector field
    $w(t)$ in  $\cH^ 1_{t0}(\Omega)^3\cap \cH^\Gamma(\tdiv \, 0)$ satisfying
    \begin{equation}\label{w}
        \curl w(t) = \mu(H(t)) H(t).
    \end{equation}
    Moreover, the mapping $\curl\colon  \cH^ 1_{t0}(\Omega)^3\cap \cH^\Gamma(\tdiv \, 0)\to  \cH(\tdiv \, 0)$ is invertible on the 
    strength of Theorem~2.9 in \cite{Ce}. In view of \eqref{local},
   the map $w$ thus belongs to $C^3 ( J_{\maxi}, \cH^1_{t0}(\Omega)^ 3\cap \cH^\Gamma(\tdiv \, 0))$. 
    Differentiating equation \eqref{w} in $t$, we deduce
    \begin{equation*}
        \curl \partial_t^ k w= \partial_t^ k (\mu(H) H) = \mu^\D(H)\partial^k_t H +g_k 
    \end{equation*}
    for $k\in\{1,2,3\}$ which proves (\ref{eq:helmholtz2}).
    Comparing this relation for $k = 1$ with  \eqref{eq:maxwell0}, we infer  $\curl (E+\partial_t w)=0$.
    Morever, the sum $E+\partial_t w$ belongs to the kernel of $\tr_t$. Theorem~2.8 of \cite{Ce} 
    now provides functions $p(t)\in \cH^1_0(\Omega)$ and $h(t)\in \cH(\tdiv \, 0)\cap \cH_0(\tcurl \, 0)$ such that
    \begin{equation}
        \label{eq:e-decomp}
        E(t)= -\partial_t w(t) +\nabla p(t) + h(t)
    \end{equation}
    for $t\in J_{\max}$.
    The spaces $\cH^\Gamma(\tdiv \, 0)$, $\nabla  \cH^1_0(\Omega)$ and $\cH(\tdiv \, 0)\cap \cH_0(\tcurl \, 0)$ are orthogonal in 
    $L^2(\Omega)^3$ and span this space, see Theorem~2.10' of \cite{Ce}. This fact furnishes the remaining regularity assertions. 
    We can now differentiate the identity \eqref{eq:e-decomp} in time, proving \eqref{eq:helmholtz1}.
\end{proof}

The energy inequality in Proposition~\ref{prop:energy} allows us to control the time integral of energy of  the electric field $E$
by the initial data and a higher order term. However,  it is necessary to bound  the time integrals of the energy of both $E$ and $H$
to obtain the desired global existence of solutions along with corresponding decay rates. 
This will be achieved by means of the Helmholtz decomposition established in Lemma~\ref{lem:helmholtz}.
We now show a lower bound for the dissipation (up to correction terms) using the quantities introduced in \eqref{def:dez}. 

\begin{prop}\label{prop:lower}
    Let the conditions of Theorem~\ref{thm:main} be satisfied.
    For $0 \le s\le t < T_*$ and $k\in \{0,1,2,3\}$, we can estimate
    \begin{equation*}
        \int_s^t e_k(\tau)\dd\tau \le c_2 \int_s^ t d_k(\tau)\dd \tau +c_3 (e_k(t)+ e_k(s)) 
        + c_4  \int_s^ t z^{3/2}(\tau) \dd \tau,
    \end{equation*}
    where the constants $c_j$ do not depend on the times $s$ and $t$.
\end{prop}

\begin{proof} 
Let $k\in \{0,1,2,3\}$. To simplify, we take $s= 0$. Equality \eqref{eq:helmholtz2} yields the identity
    \begin{equation}
        \label{eq:h-w}
        \int_{\Omega_t}  \wh\mu_k\partial^k_t H \cdot \partial^k_t H \dd (x, \tau)
            = \int_{\Omega_t} \curl \partial_t^ k w \cdot \partial^k_t H \dd (x, \tau)
            - \int_{\Omega_t} g_k\cdot \partial^k_t H \dd (x, \tau),
    \end{equation}
    where  $\Omega_t = \Omega \times (0,t)$.    Using the regularity $\partial_t^k w\in C(J_{\maxi}, \cH^1_{t0}(\Omega))^3$
    established  in Lemma~\ref{lem:helmholtz},
    we integrate by parts and then invoke the first line of the system \eqref{eq:maxwell1}.  It follows 
    \begin{align}
        \label{eq:h-w1}
        \begin{split}
        \int_{\Omega_t}& \curl \partial_t^ k w \cdot \partial^k_t H \dd (x, \tau)
        = \langle  \partial_t^ k w, \curl \partial^k_t H\rangle_{L^2((0,t), \cH^{-1}(\Omega))} \\
        &= \langle  \partial_t^ k w, \partial_t (\wh\ep_k   \partial_t^ kE)\rangle_{L^2((0,t), \cH^{-1}(\Omega))} 
        + \int_{\Omega_t} \partial_t^ k w \cdot (\sigma \partial^k_t E +\partial_t f_k) \dd (x, \tau) \\
        &= \int_\Omega \partial_t^ k w(t, \cdot)\cdot \wh\ep_k(t, \cdot)   \partial_t^ kE(t, \cdot) \dd x 
        - \int_\Omega \partial_t^ k w(0, \cdot)\cdot \wh\ep_k(0, \cdot)   \partial_t^ kE(0, \cdot) \dd x \\
        &\qquad - \int_{\Omega_t} \partial_t^{k+1} w \cdot \wh\ep_k   \partial_t^ kE \dd (x, \tau)
        +\int_{\Omega_t} \partial_t^ k w \cdot (\sigma \partial^k_t E +\partial_t f_k) \dd (x, \tau).
        \end{split}
    \end{align}
    Since $\partial_t^ k w(t, \cdot)$ belongs to  $\cH^1_{t0}(\Omega)^3\cap \cH^\Gamma(\tdiv \, 0)$, Theorem~2.9 in \cite{Ce} yields 
    the Poincar\'{e}-type estimate $\|\partial_t^ k w(\tau)\|_{L^2} \le c \,\|\curl \partial_t^ k w(\tau)\|_{L^2}$.
    From formulas \eqref{eq:helmholtz2} and \eqref{est:z}, we then infer the bound
    \begin{equation}
        \label{est:w} 
    \|\partial_t^ k w(\tau)\|_{L^2(\Omega)} \le c \,\|\curl \partial_t^ k w(\tau)\|_{L^2(\Omega)} 
     \le c\, \|\hat{\mu}_k \partial_t^k H(\tau) + g_k(\tau)\|_{L_2(\Omega) } \le ce_k^{1/2}(\tau).
    \end{equation}
    The orthogonality in equation \eqref{eq:helmholtz1} implies
    \begin{equation*}
        \|\partial_t^{k+1} w(\tau)\|_{L^2(\Omega)} \le \|\partial_t^{k} E(\tau)\|_{L^2(\Omega)}.
    \end{equation*}
    For any $\theta>0$, these inequalities along with \eqref{eq:h-w1} and \eqref{est:z} lead to the estimate
    \begin{align}
        \label{eq:h-w2}
        \Big|\int_{\Omega_t} \curl \partial_t^ k w \cdot \partial^k_t H \dd (x,\tau)\Big|
        &\le c(e_k(t)+ e_k(0)) + c\!\int_{\Omega_t} |\partial^k_t E|^2\dd (x, \tau)
        +\theta \int_{\Omega_t} \!\!|\partial^k_t w|^2 \D (x, \tau) \notag\\
        &\qquad  + c_\theta\int_{\Omega_t} |\partial^k_t E|^2 \dd (x, \tau) +  c\int_0^ t z_k^{3/2}(\tau) \dd \tau.
    \end{align}    
    As in \eqref{est:w}, we further compute
    \begin{align*}
        \int_{\Omega_t} |\partial^k_t w|^2\dd (x, \tau)
        & \le c  \int_{\Omega_t} |\curl\partial^k_t w|^2\dd (x, \tau)
        \le c  \int_{\Omega_t}  \curl\partial^k_t w \cdot \wh\mu_k^ {-1}\curl\partial^k_t w \dd (x, \tau) \\
        &=  c  \int_{\Omega_t} \curl\partial^k_t w \cdot  (\partial^k_t H +\wh\mu_k^ {-1}g_k) \dd (x, \tau)\\
        &\le  c  \,\Big|\!\int_{\Omega_t} \curl\partial^k_t w \cdot  \partial^k_t H  \dd (x, \tau)\Big| + c\int_0^t z^{3/2}(\tau)\dd \tau.
    \end{align*}
    Fixing a small number $\theta>0$, the term with $|\partial_t^k w|^2$ in equation \eqref{eq:h-w2} can now be absorbed by the left-hand 
    side and by the integral of $z^{3/2}$.  Employing also the condition $\sigma\ge \eta I$,     we arrive at
    \begin{equation*}
        \Big|\int_{\Omega_t} \curl \partial_t^ k w \cdot \partial^k_t H \dd (x, \tau)\Big|
        \le c(e_k(t)+e_k(0)) + c\int_0^t  d_k(\tau)  \dd \tau  + c\int_0^ t z_k^{3/2}(\tau) \dd \tau.
    \end{equation*}
    Equation \eqref{eq:h-w}, the last  inequality, and the estimates \eqref{est:z} yield the claim.
    Note that the constants $c$ depend neither on $t$ nor on $s$.
\end{proof}

Combining the results of Propositions \ref{prop:energy} and \ref{prop:lower}, we arrive at the following
energy bound. 
\begin{cor}
    \label{observe}
    Under the conditions of Theorem \ref{thm:main}, we have the inequality 
    \begin{equation*}
        e_k(t) + \int_s^t e_k(s) ds \leq C_1 e_k(s) + C_2 \int_s^t z^{3/2}(\tau) \dd \tau
    \end{equation*}
 for $0\le s\le t< T_*$  and $k\in\{0,1,2,3\}$, where the constants $C_k$ do not depend on $t$ and $s$.
\end{cor}
\begin{proof}
 We multiply the inequality in  Proposition~\ref{prop:lower}  by 
$\theta := \min\{c_2^{-1},(2c_3)^ {-1}\}$ and add it to  \eqref{est:energy}  from Proposition~\ref{prop:energy}, obtaining 
\begin{equation*}
    e_k(t)+ 2 \theta \int_s^ t e_k(\tau)\dd \tau \le 3e_k(s) + 2(c_1 +\theta c_4) \int_s^ t z^{3/2}(\tau) \dd \tau.\qedhere
\end{equation*}
\end{proof}

Corollary~\ref{observe} bounds the full energy  (over the  time interval $(s, t)$) by the initial energy
and superlinear higher-order energies. The quasilinear character of the equation requires to involve  
higher topological levels (up to the third order). To control these higher order terms, we need to closely investigate higher
regularity of solutions. While such an analysis has been developed in \cite{yamamoto} at the local level for the linear stationary problem,
our task is to globally extend the estimates by exploiting higher-order decay rates of the energy. 
To this end, both observability and regularity theories need to be developed -- a formidable task on its own and of independent interest.

\section{Higher-order energy observability and proof of Theorem~\ref{thm:main}}
\label{SECTION_HIGHER_ENERGY_OBSERVABILITY}

The central aim of this section is to strengthen the inequality in Corollary~\ref{observe} by including higher-order space derivatives 
represented by the terms $z(t)$ and $\int_s^t z(\tau)\dd \tau $ on the left-hand side of the estimate.
Such inequalities are often referred to as ``higher energy observability estimates'' and are used to derive decay rates of energies. 
To be more specific, the following estimate is the key step in the proof of the main result.

\begin{prop}\label{prop:z}
    Suppose the conditions of Theorem~\ref{thm:main} are satisfied.
    Then there exists a radius $\delta \in (0,\delta_0]$ such that for all 
    radii $r \in \big(0, r(\delta)\big]$ from equation \eqref{est:data}, the solutions $(E,H)$ satisfy 
    \begin{equation*}
        z(t) + \int_s^t z(\tau) \dd\tau \le C z(s)  
    \end{equation*}
for all $0\le s\le t<T_*$, where $z$ is defined in \eqref{def:dez} and the constant $C$ does not depend on time or $r\in (0,r(\delta))$,
but it depends on  $\delta$.
\end{prop}

The result easily follows from  Proposition~\ref{prop:ze} below and Corollary~\ref{observe}. 
The proof of Proposition~\ref{prop:ze} is relegated to subsequent sections. We take it for granted here.
We note that the proof of Proposition~\ref{prop:z}  actually yields a radius $\delta_1\in (0,\delta_0]$ such that the above statement is true for all 
$\delta\in (0,\delta_1]$ with a constant $C$ depending on $\delta_1$, but not on $\delta$.

\begin{proof}[Proof of Proposition~\ref{prop:z}]
Let  $0\le s\le t<T_*$. Proposition \ref{prop:ze} provides  the estimate
 \begin{equation*}
        z(t) + \int_s^t z(\tau) \dd\tau \le c_5(z(s) +  e(t) + z^2(t)) + c_6\int_s^t \big(e(\tau) + z^{3/2}(\tau)\big) \dd \tau
    \end{equation*}
for some constants $c_j$ independent of $s$ and $t$. Corollary~\ref{observe} thus yields the inequality
\begin{equation*}
        z(t) + \int_s^t z(\tau) \dd\tau \le (c_5+ C_1 (c_5+c_6))z(s) +  c_5 z^2(t) 
           + (c_6+C_2 (c_5+c_6))\int_s^t z^{3/2}(\tau) \dd \tau.
    \end{equation*}
Recall that $z(\tau) $ is bounded by $\delta^2$ on $[0, T_*)$ by \eqref{est:delta}. Fixing a sufficiently small
radius $\delta\in (0, \delta_0]$, we can now absorb the superlinear terms involving $z^2$ and $z^{3/2}$ by the left-hand side.
\end{proof}

We first discuss the linear case, which was recently treated in \cite{El} in the autonomous case.
After that we prove our main result based on  Proposition \ref{prop:z}.

\begin{rem}\label{rem:linear}
For linear material laws $\varepsilon(x, E) = \varepsilon(x)$ and $\mu(x, H ) = \mu (x)$,
one can show the variant
 \beq\label{est:linear} 
 e_0(t) + \int_s^t e_0(\tau) \dd\tau \le C e_0(s)
 \eeq
 of Proposition~\ref{prop:z} 
for \emph{all} $t\ge s \ge 0$ and all initial data, see  \cite{El}. Here we have replaced $z$ by the usual 0-th order energy $e_0$.
 This estimate easily yields the  exponential decay for all data by a standard argument.
Indeed, since \eqref{est:linear} implies $e_0(\tau)\ge C^{-1} e_0(t)$, we infer the inequality 
 \beq\label{est:linear1} 
 (1 + (t-s)C^{-1}) e_0(t) \le C e_0(s)
 \eeq
for all $t\ge s\ge 0$. Fix the time $T>0$ with $C^2/(C+T)= 1/2$. Estimate \eqref{est:linear1} then provides the bound
$e_0(nT)\le \frac12 e_0((n-1)T)$ for all $n\in\NN$. Inductively,  it follows $e_0(nT)\le 2^{-n} e_0(0)$ and hence
the exponential decay 
\beq \label{est:linear2}
e_0(t) \le M\e^{-\omega t}e_0(0)
\eeq
for suitable constants $\omega,M>0$, where we use \eqref{est:linear} once more.

Let now the coefficients $\varepsilon(t,x)$ and $\mu (t,x)$ depend on time $t\in\RR_+$.
If the supremum norms of $\partial_t\varepsilon$ and $\partial_t\mu $ are small enough, formula \eqref{eq:energy0} and the proof of 
Proposition~\ref{prop:lower} for $k=0$ imply the estimate \eqref{est:linear} also in this case. 
Then the exponential decay \eqref{est:linear2} follows as above.
\end{rem}

\emph{Proof of Theorem~\ref{thm:main}.}
We first show that $T_*=\infty$ if the radius $r>0$ in \eqref{est:data} is small enough. We suppose that 
$T_*<\infty$. Equation \eqref{eq:contr} then  yields $z(T_*)=\delta^2$, where $\delta$ is given by  Proposition~\ref{prop:z}.
On the other hand,  as in Remark~\ref{rem:linear} we deduce the inequality 
\beq\label{est:thm1} 
 (1 + (t-s)C^{-1}) z(t) \le C z(s)
 \eeq
from  Proposition~\ref{prop:z}, but now only for $0\le s\le t< T_*$ and initial data with $\|(E_0,H_0)\|_{\cH^3}^2 \le r^2$
for all radii $r\in (0, r(\delta)]$, where  $r(\delta)>0$ was introduced before
 \eqref{est:data}. The differentiated Maxwell system \eqref{eq:maxwell2} and the bounds from \eqref{est:z}
next yield 
\[z(0)\le c_0\, \|(E_0,H_0)\|_{\cH^3}^2\le   c_0 r^2\]
for a constant $c_0>0$. We now fix the radius
\begin{equation}
    \label{def:r} 
    r :=\min \Big\{r(\delta), \frac{\delta}{\sqrt{2c_0C}}\Big\}.
\end{equation}
Because of \eqref{est:thm1} with $s=0$, the number $z(t)$ is bounded by $\delta^2/2$ for $t<T_*$ and by continuity also for $t=T_*$.
This fact contradicts  $z(T_*)=\delta^2$, and hence it follows $T_*=\infty.$ We can now conclude the proof exactly as in 
Remark~\ref{rem:linear}.  \hfill $\Box$

\smallskip

In order to establish the main result of the paper, it  thus remains to prove Proposition \ref{prop:ze}. 
Necessary preparations are done in the following section.

\section{Auxiliary results}
\label{SECTION_AUXILIARY_RESULTS}

\subsection{$\mathbf{Curl}$--$\mathbf{div}$ estimates}

One can bound the $\cH^1$-norm of a field $u$ by its norms in  $\cH(\tcurl)\cap \cH(\tdiv)$ 
and the $\cH^ {1/2}$-norm of $\tr_t u$ or $\tr_n u$, see Corollary~XI.1.1 of \cite{DL}. 
In the next section, we will need a version of this result with regular, matrix-valued coefficients $a$. 
This fact does not directly follow from the case $a = I$ -- unless $a$ is scalar. It is stated 
in Remark~4 of \cite{El} with a brief indication of a proof.
For the convenience of the reader  we present a (different)  proof below.

\begin{prop}
    \label{prop:div-curl}
    Let $a \in  W^{1,\infty}(\Omega, \RR^{3+3}_{\sym})$ satisfy $a\ge \eta I$. 
    Suppose $u\in \cH(\tcurl)$ fulfills $\Div(a u)\in  L^2(\Omega)$ and $\tr_n (a u) \in \cH^{1/2}(\Gamma)$. 
    Then the vector field $u$ belongs to $\cH^1(\Omega)^3$ and fulfills
    \begin{equation*}
        \|u\|_{\cH^ 1(\Omega)} \le c\big(\|u\|_{\cH(\tcurl)} +\|\Div (a u)\|_{L^2(\Omega)}
                         + \|\tr_n(a u)\|_{\cH^{1/2}(\Gamma)}\big)=:c\kappa(u).
    \end{equation*}
\end{prop}

\begin{proof}
    There exists a finite partition of unity $\{\chi_i\}_{i}$ on $\ol{\Omega}$ such that the support
    of each $\chi_i$ is contained in a simply connected subset of $\ol{\Omega}$ with a connected $C^2$-boundary. 
    Since each $\chi_i$ is scalar, we obtain the estimate
    \begin{equation*}
        \|\chi_i u\|_{L^2(\Omega)} +\|\curl(\chi_i u)\|_{L^2(\Omega)} +\|\Div (a \chi_i u)\|_{L^2(\Omega)}
        + \|\tr_n(a \chi_i u)\|_{\cH^{1/2}(\Gamma)} \le c\kappa(u).
    \end{equation*}
    We can thus assume that $\Gamma$ is connected. In this case, $\curl u$ belongs to $\cH^\Gamma(\tdiv \, 0)$ 
    and Theorem~2.9 of \cite{Ce} yields a vector field 
    $w \in \cH^1(\Omega) \cap H_0(\tdiv \, 0)$ with $\curl u =\curl w$ and $\|w\|_{\cH^1} \le c\,\|\curl u\|_{L^2}$.
    As the difference $u-w$ is an element of $\cH(\tcurl \, 0)$, it is represented by $u-w=\nabla \ph$ for a function
    $\ph \in \cH^1(\Omega)$ by Theorem~2.8 in \cite{Ce}. We obtain
    \begin{align*} 
        \Div(a \nabla \ph)&= \Div(a u)- \Div(a w) \ \in  L^2(\Omega),\\
        \tr_n (a\nabla \ph)&=  \tr_n(a u)- \tr_n(a w) \ \in  \cH^{1/2}(\Gamma),
    \end{align*}
    because of the assumptions and the fact $w\in \cH^1(\Omega)$.
    Due to the uniform ellipticity, $\ph$ thus is an element of $\cH^2(\Omega)$ satisfying
    \begin{equation*}
        \|\ph\|_{\cH^2(\Omega)} \le c \big( \|\Div(a u)\|_{ L^2(\Omega)} + 
        \|\tr_n(a u)\|_{\cH^{1/2}(\Gamma)} +  \|w \|_{\cH^1(\Omega)} \big) \le c\kappa(u).
    \end{equation*}
    The assertion now follows from the equation $u=w+\nabla \ph$.
\end{proof}

\subsection{Geometry: coordinate transformation and differential calculus}

For a fixed distance $\rho>0$, on the collar $\Gamma_\rho = \{x\in\ol{\Omega}\,|\, \operatorname{dist}(x,\Gamma) < \rho\}$,
we can find functions $\tau^1, \tau^2, \nu\in C^4(\Gamma_\rho,\RR^3)$ such that the
vectors $\{\tau^1(x), \tau^2(x), \nu(x)\}$ form an orthonormal basis of $\RR^3$ for each point $x \in \Gamma_\rho$ 
and $\nu$ extends the outer unit normal at $\Gamma$. 
Hence, $\tau^1$ and $\tau^2$ span the tangential planes at $\Gamma$. 
For $\xi, \zeta \in \{\tau^1,\tau^2,\nu\}$, $u \in \RR^3$ and $a\in \RR^{3\times 3}$, we set
\begin{equation*}
    \partial_\xi =\sum\nolimits_j \xi_j \partial_j , \quad u_\xi= u\cdot \xi, \quad 
    u^\xi= u_\xi \xi,\quad  u^\tau = u_{\tau^1}\tau^1 +  u_{\tau^2}\tau^2,\quad  a_{\xi\zeta}= \xi^\top a \zeta.
\end{equation*}
We state several calculus formulas, which are extensively exploited in the next section. 
In the following, it is always assumed that the functions involved are sufficiently regular.
We can switch between the derivatives of the coefficient $u_\xi$ and the component $u^\xi$ up to a lower-order term since
\begin{equation*}
    \partial_\zeta u^\xi = \partial_\zeta u_\xi  \xi +  u_\xi \partial_\zeta \xi.
\end{equation*}
The commutator of tangential derivatives and traces
\begin{equation*}
    \partial_\tau \tr_t u = \partial_{\tau} (u\times \nu) =  \tr_t\partial_\tau u  + u\times  \partial_\tau \nu  \qquad \text{ on \ }\Gamma
\end{equation*}
is also of lower order. The gradient of a scalar function $\ph$ is expanded as 
\begin{equation*}
    \nabla \ph = \sum\nolimits_\xi \xi\cdot (\nabla \ph)\,\xi =  \sum\nolimits_\xi \xi \partial_\xi\ph,
\end{equation*}
so that $\partial_j = \sum_\xi \xi_j \partial_\xi$ for $j\in\{1,2,3\}$. To express the $\curl$ operator, we use the matrices
\begin{align*}
    J_1 &= \begin{pmatrix}
        0 &0 &0 \\
        0 &0 &-1 \\
        0 &1 &0
       \end{pmatrix},
        \quad
    J_2 =  \begin{pmatrix}
        0 &0 &1 \\
        0 &0 &0 \\
        -1 &0 &0
        \end{pmatrix},
        \quad
    J_3 = \begin{pmatrix}
        0 &-1 &0 \\
        1 &0 &0 \\
        0 &0 &0
       \end{pmatrix}, \\
    J(\xi) &= \sum\nolimits_j \xi_j J_j\,.
\end{align*} 
Because of 
\begin{align*}
    \CURL u &= \partial_1 [0,-u_3,u_2]^\top + \partial_2 [u_3,0, -u_1]^\top + \partial_3 [-u_2,u_1,0]^\top, 
\end{align*}
we have 
\begin{equation*}
    \curl =\sum\nolimits_j J_j\partial_j = \sum\nolimits_{j,\xi} J_j \xi_j \partial_\xi =\sum\nolimits_{\xi} J(\xi)\partial_\xi.
\end{equation*}
Observe that the kernel of $J(\nu)$ is spanned by $\nu$. 
Hence, after factoring out the null space, we can write
$J(\nu)u = J(\nu)u^\tau$, and the restriction of $J(\nu)$ to $\operatorname{span}\{\tau^1,\tau^2\}$ has an inverse $R(\nu)$.

In order to produce estimates with additional, say $\cH^1(\Omega)$--, spatial regularity, 
one typically exploits that the boundary value problem is non-characteristic. 
However, the Maxwell system is characteristic since $J(\nu)$ has the kernel $\operatorname{span}\{\nu\}$.
In order to obtain regularity in the normal direction, we employ the ``$\CURL$-$\DIV$-strategy.''
The $\CURL$ operator contains the normal derivative of the tangential components,
while the divergence condition will provide estimates  for the normal derivatives of the normal component via an ordinary differential equation.
This procedure is carried out in the next subsection.

\subsection{Representation of normal derivatives} \label{subsec:curl-div}

The following construction is based on an adaptation of the well-known ADN (Agmon--Douglis--Nirenberg) method from the elliptic theory. 
We begin by solving the equation $\curl u = f$ for normal derivatives of the tangential components of $u$.
By expanding  
\begin{equation*}
    \curl u = J(\nu) (\partial_{\nu} u)^\tau  + J(\tau^1) \partial_{\tau^1} u  + J(\tau^2) \partial_{\tau^2} u,  
\end{equation*}
we obtain
\begin{equation}
    \label{eq:curl}
    \partial_\nu u^\tau = \sum\nolimits_i (\partial_\nu \tau^i \,u_{\tau^i}+ \tau^i \partial_\nu \tau^i \cdot u)
    +  R(\nu)\Big(f - \sum\nolimits_i J(\tau^i)\partial_{\tau^i} u \Big)
\end{equation}
and hence
\begin{equation}
    \label{eq:curl1}
    \partial_\nu u^\tau = 
    R(\nu)\Big(f - \sum\nolimits_i J(\tau^i)\partial_{\tau^i} u \Big) + \mathrm{l.o.t.}(u),
\end{equation}
where $\mathrm{l.o.t.}(u)$ denote lower-order terms depending on $u$, 
but not on its derivatives.

In order to recover the normal derivative of the normal component of $u$, we resort to the divergence operator. 
The divergence of a vector field $u$ can be expressed as
\begin{equation*}
    \Div u = \sum\nolimits_j \partial_j \sum\nolimits_\xi u_\xi \xi_j 
    = \sum\nolimits_\xi \big(\partial_\xi u_\xi + \Div(\xi) u_\xi\big).
\end{equation*}
Letting $\ph=\Div(a u)$ for a matrix-valued function $a$, we derive
\begin{align}
    \label{eq:div-nu}
    \begin{split}
    \Div(a u)&= \sum\nolimits_{\xi,\zeta} \partial_\xi (\xi^\top a \zeta u_\zeta)  
         + \sum\nolimits_{\xi} \Div(\xi) \,\xi^\top  a u \\
    &= \sum\nolimits_{\xi,\zeta} (a_{\xi\zeta} \partial_\xi u_\zeta +  \partial_\xi a _{\xi\zeta} u_\zeta) 
      + \sum\nolimits_{\xi} \Div(\xi) \,\xi^ \top a u, \\
    a_{\nu\nu} \partial_\nu u_\nu  & = \ph - \sum_{(\xi,\zeta)\neq(\nu,\nu)} a_{\xi\zeta} \partial_\xi u_\zeta 
     -\sum_{\xi,\zeta} \partial_\xi a_{\xi\zeta} u_\zeta -\sum_{\xi} \Div(\xi) \,\xi^\top a u \\
    &  =:  \ph - D(a)u,
    \end{split}
\end{align}
where $D(a) u$ contains all tangential derivatives and normal derivatives of tangential components of $u$
plus lower order terms.
Next, let $a \in W^{1,\infty}(\Omega_T, \RR^{3\times 3}_{\sym})$ be uniformly positive definite,
$u \in C^1\big(J, \cH^1(\Omega)\big)^3$, and $\psi\in L^2(\Omega_J)$. In view of formula \eqref{eq:div0}, we look at the equation
\begin{equation}
    \label{eq:div1}
    \Div \big(a(t)u(t)\big)= \Div \big(a(0)u(0)\big)-\int_0^t \big(\Div(\sigma u(s)) + \psi(s)\big)\dd s
\end{equation}
for $0\le t \le T$. (In (\ref{eq:div0}) we have $\psi =0$.) We set $\gamma = \sigma_{\nu\nu}/ a_{\nu\nu}$ and  
$\Gamma(t, s)= \exp(-\int_s^t \gamma(\tau)\dd \tau)$.
Equations \eqref{eq:div-nu} and \eqref{eq:div1} yield
\begin{align*}
    a_{\nu\nu}(t) \partial_\nu u_\nu(t) &= \Div\big(a(0) u(0)\big)- D\big(a(t)\big) u(t) \\
    &\qquad -\int_0^t \big(\gamma(s) 
    a_{\nu\nu}(s) \partial_\nu u_\nu(s)+ D\big(\sigma) u(s) + \psi(s)\big)\dd s.
\end{align*}   
Differentiating with respect to $t$ and solving the resulting ODE, we obtain
\begin{align} 
    \label{eq:div-nu-sigma}
    a_{\nu\nu}(t) \partial_\nu u_\nu(t)
    &= \Gamma(t,0)a_{\nu\nu}(0)\partial_\nu u_\nu(0)
      - \int_0^t\Gamma(t,s)\big(D(\sigma)u(s)+\psi(s) +\partial_s \big(D(a(s)) u(s)\big)\big)\dd s\notag \\
    &= \Gamma(t,0)\Div(a(0) u(0))- D(a(t)) u(t)  \notag\\
   &\qquad  + \int_0^t\Gamma(t,s)\big(\gamma(s)D\big(a(s)\big)u(s)- D(\sigma)u(s) - \psi(s)\big)\dd s,
\end{align}  
where $\psi$ is the same as in  (\ref{eq:div1}) and $D(a)$ is defined in  \eqref{eq:div-nu}.

\section{A regularity result: higher order energy bounds}
\label{SECTION_REGULARITY_RESULT}

In this section we show that $\partial_t^k E$ and $\partial_t^k H$ can be bounded in $\cH^{3-k}$ for $k\in\{0,1,2\}$ by 
the $L^2$-norms of $\partial_t^k E$ and $\partial_t^k H$.  This astonishing fact is a crucial ingredient of our reasoning, 
and its proof is quite demanding.  In contrast to our situation, when studying linear autonomous  problems
such regularity estimates readily follow from  semigroup theory combined with a ``good'' characterization of 
the domains of generators and their powers -- 
the latter often is  a consequence of the theory of strongly elliptic operators. 
In our case, semigroup tools are not available. Instead we  proceed in line with the ADN approach and the $\DIV$-$\CURL$-strategy
using the techniques discussed in the previous section. We sketch the main ideas.

The $\cH^1$--norm of $\partial_t^k H$ with $k\in\{0,1,2\}$ can easily be estimated by means of the ``elliptic'' 
$\tcurl$--$\tdiv$ estimates from Proposition~\ref{prop:div-curl} because we 
control the curl and the divergence of $\partial_t^k H$ via the time differentiated Maxwell system  \eqref{eq:maxwell1} 
and \eqref{eq:div}. Aiming at higher space regularity, we can apply the above strategy to tangential derivatives
of $\partial_t^k H$ only, whereas non-tangential derivatives destroy the boundary condition in  \eqref{eq:maxwell1}.
The normal derivatives of the fields are treated similarly as in the local well-posedness theory from \cite{Sp0, Sp1, Sp2}: 
Their tangential components are read off the differentiated Maxwell system using 
the expansion \eqref{eq:curl} of the $\tcurl$-operator, while the normal components are bounded employing 
the divergence condition \eqref{eq:div} and  the formula \eqref{eq:div-nu}. In these arguments we have to restrict 
ourselves to fields localized near the boundary. The localized fields in the interior can be controlled
more easily since the boundary conditions become trivial for them.

The electric fields have less favorable divergence properties because of the conductivity term in \eqref{eq:maxwell1}. 
Instead of the $\tcurl$-$\tdiv$ estimates from Proposition~\ref{prop:div-curl}, we thus employ the energy bound
of the system \eqref{eq:maxwell-dx} that arises by differentiating the Maxwell equations in time and tangential directions.
The normal components are again treated by the $\tcurl$-$\tdiv$-strategy indicated in the previous paragraph. However,
to handle the extra divergence term in \eqref{eq:div} caused by the conductivity,  we need the more sophisticated
divergence formula \eqref{eq:div-nu-sigma} which relies on an ODE derived from \eqref{eq:div}.
This program is carried out by iteration on the space regularity. In each step one has to start with the magnetic fields
in order to use their better properties when estimating the electric ones. 


The following result is the main technical ingredient of the paper. As explained in Section~\ref{SECTION_HIGHER_ENERGY_OBSERVABILITY},
Propositions~\ref{prop:energy}, \ref{prop:lower} and \ref{prop:ze} imply  Proposition \ref{prop:z} which in turn yields our 
main Theorem~\ref{thm:main}.

\begin{prop}
    \label{prop:ze}
    We impose the conditions of Theorem~\ref{thm:main} with the exception of the simple connectedness of $\Omega$.
     Then the solutions $(E, H)$ to the Maxwell system (\ref{eq:maxwell0}) satisfy the inequality 
    \begin{equation*}
        z(t) + \int_s^t z(\tau) \dd\tau \le c_5(z(s) +  e(t) + z^2(t)) + c_6\int_s^t \big(e(\tau) + z^{3/2}(\tau)\big) \dd \tau
    \end{equation*}
    for all $0\le s\le t<T_*$, where $z$ and $e$ are defined in \eqref{def:dez} and the constants $c_j$ do not depend on $t$ and $s$.
\end{prop}

\begin{proof}
Let $(E, H)$ be a solution of \eqref{eq:maxwell0} on $J_*=[0,T_*)$ satisfying the bound \eqref{est:delta} and the 
divergence equations \eqref{eq:div0}. 
Take $k\in\ \{0,1,2\}$ and $0\le t <T_*$, where we let $s=0$ for simplicity. 
To localize the fields, we take scalar functions $\chi$ and $1 - \chi$ in $C^5(\ol{\Omega})$ having compact 
support in $\Omega \setminus \Gamma_{a/2}$ and $\Gamma_a$, respectively. 

We  already outlined our methods above. The proof is divided into several steps which we list before presenting the details.
\begin{itemize}
\item[1)] We estimate the $\cH^1$-norm of $\partial_t^k H$ using the $\CURL$-$\DIV$-estimates from Proposition~\ref{prop:div-curl}.
        
\item[2)] We bound all relevant derivatives of $E$ and $H$ in the interior using the time-space differentiated 
     Maxwell system \eqref{eq:maxwell-dx}. Here and in the other steps, highest order terms of $E$ appear on the right-hand  side which are
	absorbed later.

\item[3)] To complete the $\cH^1$-estimate, we treat $E$ near the boundary employing energy estimates for the tangential  derivatives
          and the $\tcurl$-$\tdiv$-strategy for the normal derivatives.
        These arguments rely on  the differentiated Maxwell system \eqref{eq:maxwell-tau} and variants thereof.

\item[4)] The $\cH^2$-estimates of $E$, $\partial _t E$, $H$, and $\partial_t H$ near the boundary  are carried out in a similar way,
          based on steps 1)--3).  
       
\item[5)] We handle the $\cH^3$--norm of $E$ and $H$ by iterating our techniques.  
\end{itemize}

\smallskip

\emph{1) Estimate of $\partial_t^k H$ in $\cH^1(\Omega)$.} 
These bounds are a direct  consequence of the elliptic $\CURL$-$\DIV$-estimate in Proposition~\ref{prop:div-curl}
since we can control the relevant quantities through the (time differentiated) Maxwell system \eqref{eq:maxwell1} and the divergence
equation \eqref{eq:div}. Indeed, using also the estimates  \eqref{est:z}, we obtain
\begin{align*}
    \big\|\curl \partial_t^k H(t)\big\|_{L^2(\Omega)} & \le c e_{k+1}^{1/2}(t) + cz(t)\delta_{k2},\\
    \big\|\Div (\wh\mu_k \partial_t^k H(t))\big\|_{L^2(\Omega)} &\le cz(t)\delta_{k2},\\
    \big\|\tr_n(\wh\mu_k\partial_t^k H(t))\big\|_{\cH^{1/2}(\Omega)} &\le cz(t)\delta_{k2},
\end{align*}
where $\delta_{k2}=1$ for $k=2$ and $\delta_{k2}=0$ for $k\in\{0,1\}$. Proposition~\ref{prop:div-curl} thus implies 
\begin{equation}
    \label{est:h-H1a}
    \begin{split}
    \big\|\partial_t^k H(t)\big\|_{\cH^1(\Omega)}^2 &\le ce_{k+1}(t) + cz^2(t)\delta_{k2},\\
    \int_0^t  \big\|\partial_t^k H(\tau)\big\|_{\cH^1(\Omega)}^2 \dd\tau 
    &\le c\int_0^t (e_{k+1}(\tau) + z^2(\tau)\delta_{k2})\dd\tau.
    \end{split}
\end{equation}
We stress that the inhomogeneities  in \eqref{eq:maxwell1} and  \eqref{eq:div} involving $f_2$ and $g_2$ 
are quadratic in $(E,H)$ and can thus be bounded by $z$ via \eqref{est:z}. This fact is essential in future estimates.

\smallskip

\emph{2) Estimates in the interior for $E$ and $H$}. 
We look at the localized fields  $\partial_t^k (\chi E)$ and $\partial_t^k (\chi H)$ whose support $\supp \chi$
is strictly separated from the boundary. Hence, their spatial derivatives satisfy the boundary
conditions of the Maxwell system so that we can treat the electric fields via energy bounds and the magnetic
ones via the $\tcurl$-$\tdiv$ estimates.

a) Let $\alpha\in\NN_0^3$ with $|\alpha|\le 3-k$. We apply  $\partial_x^\alpha \chi$ to the Maxwell system 
\eqref{eq:maxwell2}, deriving the equations
\begin{align}
    \label{eq:maxwell-dx}
    \ep^ \D(E)\, \partial_t \partial_x^\alpha \partial_t^k (\chi E)
    &= \curl \partial_x^\alpha\partial_t^k (\chi H) -\sigma \partial_x^\alpha\partial_t^k (\chi E) 
    + \partial_x^\alpha ([\chi,\curl]\partial_t^k H )\notag\\
    &\quad - \sum_{0\le \beta<\alpha} \binom{\alpha}{\beta} 
     \partial_x^{\alpha - \beta} (\sigma  +  \ep^ \D(E)) \,\partial_x^\beta\partial_t^k (\chi E) 
    -   \partial_x^\alpha(\chi \tilde{f}_k),  \notag \\
    \mu^ \D(H)\,\partial_t\partial_x^\alpha\partial_t^k (\chi H) &= -\curl \partial_x^\alpha\partial_t^k (\chi E) 
    -\partial_x^\alpha ([\chi,\curl]\partial_t^k E)  - \partial_x^\alpha(\chi\tilde{g}_k)\\
    &\quad - \sum_{0\le \beta<\alpha} \binom{\alpha}{\beta} 
    \partial_x^{\alpha - \beta} \mu^ \D(H) \,\partial_x^\beta\partial_t^k (\chi H), \notag\\
  \tr_t \partial_x^\alpha \partial_t^ k (\chi E) &=0, \qquad \tr_n \partial_x^\alpha \partial_t^k(\chi H)=0. \notag
\end{align}
Note that the commutator $m:=[\chi,\curl]$ is merely a multiplication operator. Lemma~\ref{lem:energy} and the 
estimates  \eqref{est:z} thus yield
\begin{align*}
    \|\partial_x^\alpha \partial_t^k &(\chi E)(t)\|_{L^2(\Omega)}^2 
    + \int_0^t \| \partial_x^\alpha \partial_t^k (\chi E)(\tau)\|_{L^2(\Omega)}^2 \dd \tau\\
    & \le c z(0) + c\int_0^t \big( z^{3/2}(\tau) +\|\partial_t^k (\chi E(\tau))\|_{\cH^ {|\alpha|-1}(\Omega)}^2
     +\|\partial_t^k (\chi H(\tau))\|^2_{\cH^ {|\alpha|-1}(\Omega)} \big)\dd \tau\\
    &\quad + c\int_{\Omega_t}\big(\partial_x^\alpha(m\partial_t^k H)\cdot  \partial_x^\alpha\partial_t^k(\chi E))
    -\partial_x^\alpha (m\partial_t^k E)\cdot  \partial_x^\alpha\partial_t^k (\chi H)\big)\dd (x, \tau).
\end{align*}
The former part of the last integral can be estimated by 
\begin{equation*}
    \frac14\int_0^t \big\| \partial_x^\alpha \partial_t^k (\chi E)(\tau)\big\|_{L^2(\Omega)}^2 \dd \tau
    + c \int_0^t \big\|\tilde \chi \partial_t^k H(\tau)\big\|_{\cH^ {|\alpha|}(\Omega)}^2 \dd \tau,
\end{equation*}
where  $\tilde{\chi}\in C_c^ \infty(\Omega \setminus \Gamma_{a/2})$ is another cut-off function being equal to 1 
on $\supp\chi$. The first summand is absorbed by the left-hand side, while the second one only involves $H$ and 
can be treated separately. The latter part of the integral on $\Omega_t$ is similarly bounded by
\begin{equation*}
    \theta\int_0^t \big\|\partial_t^k E(\tau)\big\|_{\cH^ {|\alpha|}(\Omega)}^2 \dd \tau
    + c(\theta) \int_0^t \big\| \tilde \chi \partial_t^k H(\tau)\big\|_{\cH^ {|\alpha|}(\Omega)}^2 \dd \tau
\end{equation*}
for an arbitrary (small) $\theta >0$. It follows
\begin{align}\label{est:e-int}
    \big\|\partial_x^\alpha \partial_t^k (\chi E)(t)&\big\|_{L^2(\Omega)}^2 
    + \int_0^t \big\| \partial_x^\alpha \partial_t^k (\chi E)(\tau)\big\|_{L^2(\Omega)}^2 \dd \tau\\
   &\le cz(0) + c\int_0^t(z^{3/2}(\tau)
       +  \big\|\partial_t^k (\chi E(\tau)\big\|_{\cH^ {|\alpha|-1}(\Omega)}^2)\dd\tau\notag\\
    &\qquad   + \theta\int_0^t \big\|   \partial_t^k E(\tau)\big\|_{\cH^{|\alpha|}(\Omega)}^2 \dd \tau
    + c(\theta) \int_0^t \big\| \tilde \chi \partial_t^k H(\tau)\big\|_{\cH^ {|\alpha|}(\Omega)}^2 \dd \tau.\notag
\end{align}

b) To treat $H$, we only need to look at the case $|\alpha|\le 2-k$. Equations \eqref{eq:div0} and \eqref{eq:div} 
yield
\begin{align}
    \label{eq:div-chi-h}
    \Div\big(\wh\mu_k \partial_x^\alpha \partial_t^k (\chi H)\big) 
    &= \partial_x^\alpha ([\Div, \chi] \wh\mu_k \partial_t^k H) - \sum_{0\le \beta <\alpha} \binom{\alpha}{\beta} 
    \Div\big(\partial_x^ {\alpha-\beta }\wh\mu_k\, \partial_x^\beta(\partial_t^k (\chi H)\big)
    -\partial_x^\alpha (\chi g_k).
\end{align}
Recalling formulas \eqref{eq:maxwell-dx} and \eqref{est:z}, we deduce
\begin{align*} 
    \big\|\curl \partial_x^\alpha\partial_t^k& (\chi H(t))\big\|_{L^2(\Omega)} 
    + \big\|\Div \partial_x^\alpha\partial_t^k (\chi H(t))\big\|_{L^2(\Omega)} \\
    &\le c \Big(\!z(t) + \big\|\partial_t^k \tilde{\chi} H(t)\big\|_{\cH^ {|\alpha|}(\Omega)}
    + \big\|\partial_t^{k+1}(\chi E(t))\big\|_{\cH^ {|\alpha|}(\Omega)} 
    + \big\|\partial_t^k (\chi E(t))\big\|_{\cH^{|\alpha|}(\Omega)}\Big)
 \end{align*} 
Proposition~\ref{prop:div-curl} now implies the inequalities
\begin{align} 
    \label{est:h-int}
    \big\|\partial_t^k \chi H(t)\big\|_{\cH^{|\alpha|+1}(\Omega)}^2
    &\le c \big(z^{2}(t) + \big\|\partial_t^k \tilde{\chi} H(t)\big\|_{\cH^ {|\alpha|}(\Omega)}^2
    + \max_{j \le k+1}\|\partial_t^j (\chi E(t))\|_{\cH^ {|\alpha|}(\Omega)}^2 \big),\\
    \int_0^t \big\|\partial_t^k \chi H(\tau)\big\|_{\cH^{|\alpha|+1}(\Omega)}^2\dd \tau
    & \le c \int_0^t \big(z^{2}(\tau) + \big\|\partial_t^k \tilde{\chi} H(\tau)\big\|_{\cH^ {|\alpha|}(\Omega)}^2
    + \max_{j \le k+1} \big\|\partial_t^j (\chi E(\tau))\big\|_{\cH^ {|\alpha|}(\Omega)}^2 \big) \dd\tau.\notag
\end{align}
Here, we can replace $\chi$ by $\tilde{\chi} $ from inequality \eqref{est:e-int} and $\tilde{\chi}$ by a function 
$\breve{\chi}\in C_c^\infty( \Omega \setminus \Gamma_{a/2})$ which is equal to 1 on $\supp \tilde{\chi}$.

We set $y_j(t) =\max_{0\leq k\le 3-j} \big\|\partial_t^k \chi (E(t),H(t))\big\|_{\cH^j}^2$.
The estimates \eqref{est:h-H1a}, \eqref{est:e-int} and \eqref{est:h-int} iteratively imply 
\begin{align}
    \label{est:int}
    y_j(t) + \int_0^t y_j(\tau) \dd\tau &\le  cz(0) +  c \big(e(t) + z^2(t)\big)
    +c(\theta)\int_0^t (e(\tau)+ z^ {3/2}(\tau))\dd \tau\notag\\
    &\qquad + \theta \max_{0\le k\le 3-j} \int_0^t \|\partial_t^k E(\tau)\|_{\cH^j(\Omega)}^2 \dd\tau
\end{align}
for any $\theta>0$ and $j\in\{1,2,3\}.$

\smallskip

\emph{3) Boundary-collar estimate of $\partial_t^k E$ in $\cH^1$.}
a) We write $\hat\chi = 1-\chi$ and $\partial_\tau = (\partial_{\tau^1}, \partial_{\tau^2})$. Let $\alpha\in\NN_0^2$ 
with $0<|\alpha| \le 3-k$. (For the later use, also higher-order space derivatives are treated.)  We localize the system 
near the boundary by including the cut-off $\hat\chi$ into the equations \eqref{eq:maxwell2}, and then apply 
$\partial_\tau^\alpha$ to the resulting system. The localized tangential-time derivatives of $(E,H)$  thus satisfy 
 \begin{align}   \label{eq:maxwell-tau}
\ep^ \D(E)\, \partial_t \partial_\tau^\alpha \partial_t^k (\hat\chi E)
  &= \curl \partial_\tau^\alpha\partial_t^k (\hat\chi H) -\sigma \partial_\tau^\alpha\partial_t^k (\hat\chi E) 
+[\partial_\tau^\alpha,\curl]\partial_t^k(\hat\chi H) + \partial_\tau^\alpha ([\hat\chi,\curl]\partial_t^k H)\notag\\
    &\quad - \sum_{0\le \beta<\alpha} \binom{\alpha}{\beta} 
    \partial_\tau^{\alpha - \beta} (\sigma  +  \ep^ \D(E)) \,\partial_\tau^\beta\partial_t^k (\hat\chi E) 
    -   \partial_\tau^\alpha(\hat\chi \tilde{f}_k),  \notag\\
    \mu^\D(H)\,\partial_t\partial_\tau^\alpha\partial_t^k (\hat\chi H) 
    &= -\curl \partial_\tau^\alpha\partial_t^k (\hat\chi E) 
      -\partial_\tau^\alpha ([\hat\chi,\curl]\partial_t^k E)  
      - [\partial_\tau^\alpha ,\curl]\partial_t^k (\hat\chi E) \\
    &\quad - \sum_{0\le \beta<\alpha} \binom{\alpha}{\beta} 
     \partial_\tau^{\alpha - \beta} \mu^ \D(H) \,\partial_\tau^\beta\partial_t^k (\hat\chi H) 
     - \partial_\tau^\alpha(\hat\chi\tilde{g}_k), \notag\\
    \tr_t \partial_\tau^\alpha \partial_t^ k (\hat\chi E) 
    &= [ \partial_\tau^\alpha, \tr_\tau] \partial_t^ k (\hat\chi E)=:\chi. \notag
\end{align}
The commutators $[\partial_\tau^\alpha ,\curl]$ are differential operators of order $| \alpha|$ with bounded 
coefficients, whereas $[\partial_{\tau}^{\alpha}, \tr_{\tau} ]$ is of  order $|\alpha |-1 $ 
on the boundary and hence a bounded operator from $\cH^{|\alpha|-1/2}(\Gamma)$ to $\cH^{1/2}(\Gamma) $.
%
We now use the energy identity in Lemma~\ref{lem:energy} with
$a =\ep^ \D(E)$, $b = \mu^\D(H)$, $u=\partial_\tau^\alpha \partial_t^k (\hat\chi E)$, and
$v = \partial_\tau^\alpha\partial_t^k (\hat\chi H)$. 
The commutator terms, the sums and the summands with $f_k$ and $g_k$ yield the inhomogeneities
$\ph$ and $\psi$, respectively. 
From Lemma~\ref{lem:energy} we deduce the inequality
\begin{align}\label{est:rhs-3}
 \big\|&\partial_\tau^\alpha \partial_t^k (\hat \chi E)(t)\big\|_{L^2(\Omega)}^2 
    + \int_0^t \big\|\partial_\tau^\alpha \partial_t^k (\hat\chi E)(\tau)\big\|_{L^2(\Omega)}^2 \dd \tau\notag\\
&\le cz(0) + c\int_{\Omega_t}\big(|\partial_t a u\cdot u| + |\partial_t b v\cdot v| +|\ph\cdot u|+|\psi\cdot v|\big)\dd (\tau,x) 
  +  c\int_{\Gamma_t} |\chi \cdot \tr_\tau v|\dd (\tau,x).
\end{align}
Several terms on the right-hand side are super-quadratic in $(E,H)$ and can be bounded by  $cz^{3/2}$ due to 
\eqref{est:z}. The quadratic ones need more care. The summands in $\ph\cdot u$ and $\psi\cdot v$ containing the 
commutators are less or equal to
\[\theta\int_0^t \big\|\partial_t^k E(\tau)\big\|_{\cH^{|\alpha|}(\Omega)}^2 \dd \tau
    + c(\theta) \int_0^t \big\|\tilde \chi \partial_t^k H(\tau)\big\|_{\cH^ {|\alpha|}(\Omega)}^2 \dd \tau\]
with any (small) constant $\theta>0$ and a cut-off $\tilde{\chi}\in C_c^\infty(\Gamma_a)$ being equal to 1 on $\supp \hat\chi$.
The boundary integral is estimated by the same expression, where we
use the dual paring $H^ {1/2}(\Gamma)\times H^ {-1/2}(\Gamma)$ and that $\partial_{\tau^i}$ belongs to 
$\cB(H^{1/2}(\Gamma), H^{-1/2}(\Gamma))$. The sums over $\beta$  give rise to the terms
\[\frac14\int_0^t \big\|\partial_\tau^\alpha(\partial_t^k \hat\chi E(\tau))\big\|_{L^2(\Omega)}^2 \dd \tau
    + c \int_0^t \big\|\hat \chi \partial_t^k E(\tau)\big\|_{\cH^ {|\alpha|-1}(\Omega)}^2 \dd \tau
     + c \int_0^t \big\|\hat \chi \partial_t^k H(\tau)\big\|_{\cH^ {|\alpha|}(\Omega)}^2 \dd \tau  \]
plus super-quadratic terms. We thus arrive at 
\begin{align}  \label{est:e-tau}
    \big\|\partial_\tau^\alpha \partial_t^k &(\hat \chi E)(t)\big\|_{L^2(\Omega)}^2 
    + \int_0^t \big\|\partial_\tau^\alpha \partial_t^k (\hat\chi E)(\tau)\big\|_{L^2(\Omega)}^2 \dd \tau\\
    &\le cz(0) + c(\theta) \int_0^t \big(\big\|\hat\chi\partial_t^k E(\tau)\big\|_{\cH^{|\alpha|-1}(\Omega)}^2
   + \big\|\tilde \chi \partial_t^k H(\tau)\big\|_{\cH^ {|\alpha|}(\Omega)}^2\big) \dd \tau \notag\\
 &\qquad   + \theta\int_0^t \big\|\partial_t^k E(\tau)\big\|_{\cH^{|\alpha|}(\Omega)}^2 \dd \tau
     + c\int_0^tz^{3/2}(\tau)\dd \tau. \notag
\end{align}

b) In order to finalize the $\cH^1$-estimate for the electric field, we must control the normal derivatives. 
Their tangential component is  determined by the $\CURL$-term in the Maxwell system. More precisely,
the second equation in  \eqref{eq:maxwell2}, formula  \eqref{eq:curl} and the estimate \eqref{est:z} imply 
\begin{equation}
    \label{est:e-nu-tau}
    \big\|\partial_\nu(\partial_t^k (\hat \chi E(t))^\tau \big\|_{L^2(\Omega)}^ 2\le c\big( e_{k+1}(t) + z^2(t)
    + \big\|\partial_\tau \partial_t^k (\hat \chi E(t))\big\|_{L^2(\Omega)}^ 2\big).
\end{equation}

For the normal component we use the $\DIV$-relations, where we also consider higher tangential derivatives
for later use. We first look at the case $k\in\{1,2\}$ and apply $\partial_{\tau}^\alpha \hat{\chi}$
to equation \eqref{eq:div} with $|\alpha|\le 2-k$. It follows
\begin{align}
    \label{eq:div-tau}
    \Div \big(\ep^ \D(E)\partial_\tau^\alpha \partial_t^ k(\hat\chi E)\big) &= - D(\ep^\D(E), \alpha) \partial_t^kE 
    - \Div(\sigma \partial_\tau^\alpha (\hat\chi \partial_t^{k-1}E)) \\
    &\qquad   -D(\sigma,\alpha)\partial_t^{k-1}E- \partial_\tau^\alpha( \hat\chi \Div f_k \big).\notag
\end{align}
Here we abbreviate the commutator terms
 \begin{equation*}
    D(a, \alpha) u := \partial_\tau^\alpha \big([\hat\chi,\Div] (au)\big) + [\partial_\tau^ \alpha,\Div](\hat\chi au)
    + \sum_{0\le \beta <\alpha} \binom{\alpha}{\beta} 
   \Div\big(\partial_\tau^ {\alpha-\beta }a\, \partial_\tau^\beta(\hat\chi u)\big)
\end{equation*}
for a matrix-valued function $a$ and a vector function $u$. 
Observe that $D(a, \alpha)$ is a differential operator of order $|\alpha|$ and that $|D(a,0)u|\le c\,|u|$. 
Below we treat the equality \eqref{eq:div-tau} by means of  formula \eqref{eq:div-nu}. For $k=0$ 
 the divergence equation contains a time integral
and initial data which are handled by means of identity \eqref{eq:div-nu-sigma}. To avoid terms which 
grow linearly in time, we have to derive another equation from \eqref {eq:maxwell0}, namely,
 \begin{align}   \label{eq:maxwell-tau-ep}
\partial_t (\ep(E)\partial_\tau^\alpha  (\hat\chi E))
  &= \curl \partial_\tau^\alpha (\hat\chi H) -\sigma \partial_\tau^\alpha (\hat\chi E) 
 -[\curl,\partial_\tau^\alpha](\hat\chi H) - \partial_\tau^\alpha ([\curl,\hat\chi] H) \notag\\
    &\quad - \sum_{0\le \beta<\alpha} \binom{\alpha}{\beta} 
    \partial_\tau^{\alpha - \beta} (\sigma  +  \ep(E)) \,\partial_\tau^\beta (\hat\chi E).
\end{align}
Writing $h$ for the sum of the three commutator terms, we derive the divergence relation 
\begin{align}\label{eq:div-tau0}
\Div \big(\ep(E(t)) \partial_\tau^\alpha (\hat\chi E(t))\big)
    = \Div \big(\ep(E_0)\partial_\tau^\alpha(\hat\chi E_0)\big) 
   - \int_0^t \big(\Div(\sigma \partial_\tau^\alpha (\hat\chi E(\tau))) + \Div h(\tau)\big)\dd\tau. 
\end{align}

c) To control $\partial_\nu E_\nu$, we use equation \eqref{eq:div-tau0} with $\alpha=0$ and identity \eqref{eq:div-nu-sigma},
where we put $a = \ep(E)$, $u=\hat\chi E$,  and $\psi= \Div h$. The function 
$\gamma=\sigma_{\nu\nu}/a_{\nu\nu}$ is bounded from below by $\gamma_0 = c\eta>0$. We then get the estimate
\begin{align*} 
    \big\|&\partial_\nu (\hat \chi E(t))_\nu\big\|_{L^2(\Omega)}^2\\
        &\le c\e^{-\gamma_0 t} z(0) + c\,\big(\|E(t)\|_{L^2(\Omega)}^2  + \|\partial_\tau (\hat \chi E(t))\|_{L^2(\Omega)}^2 
    + \|\partial_\nu (\hat \chi E(t))^\tau\|_{L^2(\Omega)}^ 2\big)\notag\\
    &\quad + c\int _0^t \e^{-\gamma_0 (t-\tau)}\big(\|E(\tau)\|_{L^2}^2  
    + \|\partial_\tau (\hat \chi E(\tau))\|_{L^2}^2 
    + \|\partial_\nu (\hat \chi E(\tau))^\tau\|_{L^2}^2+ \|H(\tau)\|_{\cH^1}^2 + z^2(\tau)\big)\dd s. \notag
\end{align*}
This bound together with equations \eqref{est:e-tau}, \eqref{est:e-nu-tau} and \eqref{est:h-H1a} now implies
\begin{align} 
    \label{est:e-nu-nu0}
    \big\|&\partial_\nu (\hat \chi E(t))_\nu\big\|_{L^2(\Omega)}^2 
    + \int _0^t \big\|\partial_\nu (\hat \chi E(s))_\nu\big\|_{L^2(\Omega)}^2\dd s\\
    &\le c\big(z(0)+e(t)+z^2(t)\big) + \theta \!\int_0^t  \|E(s)\|_{\cH^1(\Omega)}^2\dd s
    + c(\theta)\!\int_0^t \big(e(s) +z^{3/2}(s)\big)\dd s, \notag
\end{align}
where the small number $\theta$ comes from \eqref{est:e-tau}.
Combining \eqref{est:e-tau}, \eqref{est:e-nu-tau}, \eqref{est:e-nu-nu0} and \eqref{est:h-H1a}, we conclude
\begin{align*}
    \big\|&\hat \chi E(t)\big\|_{\cH^1(\Omega)}^2 
    + \int _0^t \big\|\hat \chi E(s)\big\|_{\cH^1(\Omega)}^2\dd s\\
    &\le c\big(z(0)+e(t)+z^2(t)\big) + \theta \int_0^t \|E(s)\|_{\cH^1(\Omega)}^2\dd s
    + c(\theta)\int_0^t \big(e(s) +z^{3/2}(s)\big)\dd s.\notag
\end{align*}

For $k\in\{1,2\}$ we proceed similarly using equation \eqref{eq:div-tau} with $\alpha=0$ and formula  
\eqref{eq:div-nu}  for the normal component.  Here the term $\|\partial_t^{k-1}\hat\chi E(t)\|_{\cH^1(\Omega)}^2$ 
appears on the right-hand  side,  which can be treated iteratively. We thus show the inequality
\begin{align} 
    \label{est:e-H1}
    \big\|\partial_t^k \hat \chi& E(t)\big\|_{\cH^1(\Omega)}^2 
    + \int _0^t \big\|\partial_t^k \hat \chi E(s)\big\|_{\cH^1(\Omega)}^2\dd s\\
    &\le c\big(z(0)+e(t)+z^2(t)\big) + \theta\! \int_0^t \big\|\partial_t^k E(s)\big\|_{\cH^1(\Omega)}^2\dd s
    + c(\theta)\!\int_0^t \big(e(s) +z^{3/2}(s)\big)\dd s\notag
\end{align}
for all $k\in\{0,1,2\}$. Both in this relation and in inequality \eqref{est:e-int} for $|\alpha|=1$, we now choose  
a sufficiently small $\theta>0$. Together with \eqref{est:h-H1a}, we derive our first order bound.
\begin{lemma}[$\cH^1$-estimate]\label{l:H1}
Let $k\in \{0,1,2\}$ and the assumptions of Theorem~\ref{thm:main} with the exception of the simple connectedness 
of $\Omega$ be satisfied.  Then we can estimate
    \begin{align} \label{est:H1}
        \big\|\partial_t^k  (E(t), H(t)\big)&\|_{\cH^1(\Omega)}^2 
        + \int _0^t \big\|\partial_t^k (E(s),H(s))\big\|_{\cH^1(\Omega)}^2\dd s\\
        &\le c\big(z(0)+e(t)+z^2(t)\big) +    c\int_0^t \big(e(s) +z^{3/2}(s)\big)\dd s, \notag
    \end{align}
    where the constant $c$ does not depend on time $t$.
    
\end{lemma}

\emph{4) Estimate  in $\cH^2$.}
While the bound of $H$ in $\cH^1$ was entirely based on the $\tcurl$-$\tdiv$-estimates of 
Proposition~\ref{prop:div-curl}, this is only partly possible in $\cH^2$ or $\cH^3$ since normal 
derivatives violate the boundary conditions.  We thus have to employ the $\curl$-$\tdiv$ strategy of 
Subsection~\ref{subsec:curl-div} also for $H$,  proceeding in multiple steps. We let $k\in\{0,1\}$.

a) We first control tangential space-time derivatives of $H$ in $\cH^1$ by means of 
$\tcurl$-$\tdiv$ estimates. Proposition~\ref{prop:div-curl} yields
\begin{eqnarray} \label{eq:H-H2} 
 \big\|\partial_\tau  \partial_t^ k \hat{\chi} H\big\|_{\cH^1(\Omega)} 
    \leq c\,\big(\|\CURL \partial_\tau  \partial_t^ k \hat{\chi}H\|_{L^2(\Omega)} 
        + \|\DIV \partial_\tau  \partial_t^ k \hat{\chi}H\|_{L^2(\Omega)} 
        + \|\tr_n \partial_\tau  \partial_t^ k \hat{\chi}H\|_{H^{1/2}(\Gamma)}\big).
\end{eqnarray}
From equations \eqref{eq:maxwell1}, \eqref{eq:div0} and \eqref{eq:div} we deduce
\begin{align*}
\tr_n\!\big(\wh\mu_k \partial_\tau \partial_t^ k(\hat\chi H)\big) 
 &=[\tr_n,\partial_\tau](\partial_t^k\hat\chi H) - \tr_n\big(\partial_\tau \wh\mu_k\,\partial_t^k(\hat\chi H)\big),\\
\Div\!\big(\wh\mu_k \partial_\tau \partial_t^ k(\hat\chi H)\big)
  &= \partial_\tau([\Div,\hat\chi]\wh\mu_k\partial_t^k H) - [\partial_\tau,\Div](\wh\mu_k\partial_t^ k(\hat\chi H))
                    - \Div(\partial_\tau \wh\mu_k \, \partial_t^ k(\hat\chi H)). \notag
\end{align*}
The  commutator $[\partial_{\tau}, \DIV]$ is of order one and the others are of order zero. For the 
$\tcurl$-relation  we can use the first equation in \eqref{eq:maxwell-tau} with $|\alpha|=1$.  By means of \eqref{est:z}, 
we estimate 
\begin{align}\label{div-curlbounds}
 \|\Div\!\big(\wh\mu_k \partial_\tau \partial_t^ k(\hat\chi H(t))\big)\|_{L^2(\Omega) } 
       &\leq c\, \| \partial_t^k H(t) \|_{\cH^1 (\Omega) }, \notag  \\
 \|\CURL\!\big(\partial_\tau  \partial_t^ k \hat{\chi}H(t)\big)\|_{L^2(\Omega) } 
       &\leq c\, \big( \|\partial_t^{k+1} E(t)\|_{\cH^1(\Omega) } + \|\partial_t^k (E(t),H(t)) \|_{\cH^1(\Omega) } \big), \\    
 \|\tr_n\!\big(\wh\mu_k \partial_\tau \partial_t^ k(\hat\chi H(t))\big)\|_{\cH^{1/2}(\Gamma)} 
        &\leq c\,\|\partial_t^k H(t)\|_{\cH^1(\Omega) }.\notag
\end{align}
Since $k+1\le 2$, inequalities \eqref{est:H1}, \eqref{eq:H-H2} and \eqref{div-curlbounds} now imply
\begin{align} 
    \label{est:h-Dtau-H1}
    \big\|\partial_\tau \partial_t^k  (\hat\chi H(t))&\big\|_{\cH^1(\Omega)}^2 
    + \int _0^t \big\|\partial_\tau\partial_t^k (\hat\chi H(s))\big\|_{\cH^1(\Omega)}^2\dd s\\
    &\le c\big(z(0)+e(t)+z^{2}(t)\big) +    c\int_0^t \big(e(s) +z^{3/2}(s)\big)\dd s.\notag
\end{align}

b) The estimate of the normal derivative of $H$ in $\cH^1$ will be based on the $\curl$-$\tdiv$ strategy.
We first solve in the first equation of \eqref{eq:maxwell-tau} with $\alpha =0$ for 
$\partial_\nu (\partial_t^k \hat\chi H(t))^\tau$ using formula \eqref{eq:curl}, and derive the inequality
\begin{equation*}
    \|\partial_\nu \partial_t^k(\hat\chi H(t))^\tau\|_{\cH^1(\Omega)}
     \le c\,\big(  \|\partial_\tau \partial_t^k(\hat\chi H(t))\|_{\cH^1(\Omega)} 
         + \|\partial_t^k(E(t),H(t))\|_{\cH^1(\Omega)} + \|\partial_t^{k+1}E(t)\|_{\cH^1(\Omega)}\big).
\end{equation*}       
Equations \eqref{est:H1} and  \eqref{est:h-Dtau-H1} thus allow us to bound  the tangential components by
\begin{align} 
    \label{est:h-Dnu-tau-H1}
    \big\|\partial_\nu \partial_t^k(\hat\chi H(t))^\tau&\big\|_{\cH^1(\Omega)}^2 
   + \int _0^t \big\|\partial_\nu \partial_t^k (\hat\chi H(s))^\tau\big\|_{\cH^ 1(\Omega)}^2\dd s\\
    &\le c\big(z(0)+e(t)+z^{2}(t)\big) + c\int_0^t \big(e(s) +z^{3/2}(s)\big)\dd s.\notag
\end{align}

As to the normal component, we apply formula \eqref{eq:div-nu} to the divergence equation \eqref{eq:div-chi-h} 
with $\alpha=0$ and $\hat\chi$ instead of $\chi$. 
The $\cH^1$-norm of $\partial_\nu (\hat\chi H(t))_\nu$ is thus controlled by that of 
$\hat\chi H(t)$, $\partial_\tau (\hat\chi H(t))$, and $\partial_\nu (\hat\chi H(t))^\tau$.
From \eqref{est:H1}, \eqref{est:h-Dtau-H1} and \eqref{est:h-Dnu-tau-H1}, we now conclude
\begin{align} 
    \label{est:h-Dnu-nu-H1}
    \big\|\partial_\nu  \partial_t^k (\hat\chi H(t))_\nu&\big\|_{\cH^1(\Omega)}^2 
    + \int _0^t \big\|\partial_\nu  \partial_t^k (\hat\chi H(s))_\nu\big\|_{\cH^ 1(\Omega)}^2\dd s\\
    &\le c \big(z(0)+e(t)+z^2(t)\big) + c\int_0^t \big(e(s) +z^{3/2}(s)\big)\dd s.\notag 
\end{align}
Combining the inequalities \eqref{est:h-Dtau-H1}, \eqref{est:h-Dnu-tau-H1}, \eqref{est:h-Dnu-nu-H1}, 
\eqref{est:h-int} and \eqref{est:H1}, we arrive at the $\cH^2$-estimate for the fields $H$ and $\partial_t H$
\begin{align} 
    \label{est:h-H2}
    \big\|\partial_t^k H(t)&\big\|_{\cH^2(\Omega)}^2 
    + \int _0^t \big\|\partial_t^k H(s)\big\|_{\cH^2(\Omega)}^2\dd s\\
    &\le c\big(z(0)+e(t)+z^2(t)\big) + c\int_0^t \big(e(s) +z^{3/2}(s)\big)\dd s.\notag
\end{align}

c) We now turn our attention to $E$. Let $|\alpha|=2$. The $L^2$-norms of the tangential derivatives 
$\partial_\tau^\alpha(\hat\chi\partial_t^k E)$ is already controlled
 via  inequalities \eqref{est:e-tau}, \eqref{est:H1}, and \eqref{est:h-H2} up to the term
\begin{equation*}
    \theta\int_0^t \big\| \partial_t^k E(\tau)\big\|_{\cH^{2}(\Omega)}^2 \dd \tau.
\end{equation*}
The second equation in \eqref{eq:maxwell-tau} with $|\alpha|=1$ and formula  \eqref{eq:curl} 
lead to the estimate
 \[\big\|\partial_\nu \big[\partial_\tau \partial_t^k(\hat\chi E(t))\big]^\tau\big\|_{L^2(\Omega)}
    \le c\big[ \|\partial_\tau^2 \partial_t^k(\hat\chi E(t))\|_{L^2(\Omega)}
      + \|\partial_t^k (E(t),\!H(t))\|_{\cH^1(\Omega)} + \|\partial_t^{k+1} E(t)\|_{\cH^1(\Omega)} + z(t)\big].\]
Combined with the above mentioned tangential  bound and the $\cH^1$--result  \eqref{est:H1}, we obtain 
\begin{align} \label{est:De-nu1-tau}
    \big\|&\partial_\nu \big(\partial_\tau \partial_t^k(\hat \chi E(t) \big)^\tau \big\|_{L^2}^2 
    + \big\|\partial_\tau^\alpha \partial_t^k(\hat\chi E(t))\big\|_{L^2}^2 
  + \int_0^t\Big[\big\|\partial_\nu \big(\partial_\tau \partial_t^k (\hat \chi E(s))\big)^\tau \big\|_{L^2}^2
     + \big\|\partial_\tau^\alpha \partial_t^k (\hat\chi E(s))\big\|_{L^2}^2\Big]\D s \notag\\
    &\le c(z(0)+e(t)+z^2(t)) + \theta\int_0^t  \| \partial_t^k E(s)\|_{\cH^2(\Omega)}^2\dd s
    + c(\theta) \int_0^t \big(e(s) +z^{3/2}(s)\big)\dd s 
\end{align}

d) For the normal component and $k=0$, we look at the divergence relation \eqref{eq:div-tau0}  with $|\alpha|=1$. 
As in \eqref{est:e-nu-nu0}, we deduce from \eqref{eq:div-nu-sigma} the estimate
\begin{align}     \label{est:e-nu2-tau}
    \big\|\partial_\nu \big( &\partial_\tau(\hat \chi E(t))\big)_\nu\big\|_{L^2(\Omega)}^2 
    + \int _0^t \big\|\partial_\nu \big(\partial_\tau(\hat \chi E(s))\big)_\nu \big\|_{L^2(\Omega)}^2\dd s\\
    &\le c(z(0)+e(t)+z^2(t)) + \theta \int_0^t  \| E(s)\|_{\cH^2(\Omega)}^2\dd s
    + c(\theta)\int_0^t  \big(e(s) +z^{3/2}(s)\big)\dd s. \notag
\end{align}
The two above inqualities imply 
\begin{align}
    \label{est:e-tau-H1} 
    \big\|\partial_\tau(&\hat \chi E(t))\big\|_{\cH^1(\Omega)}^2 
    + \int _0^t \big\|\partial_\tau(\hat \chi E(s))\big\|_{\cH^1(\Omega)}^2\dd s\\
    &\le c\big(z(0)+e(t)+z^2(t)\big) + \theta \int_0^t  \| E(s)\|_{\cH^2(\Omega)}^2\dd s
    + c(\theta)\int_0^t  \big(e(s) +z^{3/2}(s)\big)\dd s.\notag
\end{align}

To treat the case $k=1$, we start from  the divergence equation \eqref{eq:div-tau} with $|\alpha|=1$
and use formula \eqref{eq:div-nu}. Employing also estimates \eqref{est:De-nu1-tau}, \eqref{est:e-tau-H1} and \eqref{est:z}, we get
\begin{align} \label{est:De-nu2-tau} 
\big\|&\partial_\nu \big(\partial_\tau(\hat \chi \partial_t E(t))\big)_\nu\big\|_{L^2(\Omega)}^2 
 + \int _0^t \big\|\partial_\nu \big(\partial_\tau(\hat\chi\partial_t E(s))\big)_\nu\big\|_{L^2(\Omega)}^2\dd s\\
    &\le c(z(0)+e(t)+z^2(t)) + \theta \int_0^t  \big(\| E(s)\|_{\cH^2(\Omega)}^2
    +   \|\partial_t E(s)\|_{\cH^2(\Omega)}^2\big)\dd s
    + c(\theta)\int_0^t  \big(e(s) +z^{3/2}(s)\big)\dd s.\notag
\end{align}
Combined with inequality \eqref{est:De-nu1-tau}, this relation leads to 
\begin{align}
    \label{est:De-tau-H1} 
    \| &\partial_\tau\partial_t(\hat \chi E(t))\|_{\cH^1(\Omega)}^2 
    + \int _0^t \|\partial_\tau\partial_t(\hat \chi E(s))\|_{\cH^1(\Omega)}^2\dd s\\
    &\le c(z(0)+e(t)+z^2(t)) + \theta \int_0^t \big(\| E(s)\|_{\cH^2(\Omega)}^2+
       \|\partial_t E(s)\|_{\cH^2(\Omega)}^2\big)\dd s
    + c(\theta)\int_0^t  \big(e(s) +z^{3/2}(s)\big)\dd s.\notag
\end{align}

e) It remains to control the term $\partial_\nu^2 (\partial_t^k\hat\chi E)$. 
We first replace in system \eqref{eq:maxwell-tau} the derivative $\partial_\tau^\alpha$ by $\partial_\nu$.
The resulting second equation, the $\tcurl$-formula  \eqref{eq:curl} and estimates \eqref{est:z} allow us to bound 
\[\big\|\partial_\nu\big(\partial_\nu\partial_t^k(\hat\chi E(t))\big)^\tau\|_{L^2(\Omega)}
\le c\,\big(  \big\|\partial_\tau\partial_\nu\partial_t^k(\hat\chi E(t))\big\|_{L^2(\Omega)}
 + \max_{j\le 2}\big\|\partial_t^j(E(t), H(t))\big\|_{\cH^1(\Omega)} +z(t)\Big). \]
 The right-hand side can be controlled via inequalities  \eqref{est:H1}, \eqref{est:e-tau-H1},  and  \eqref{est:De-tau-H1}.

For the normal component we use the modifications of the divergence relations \eqref{eq:div-tau0} and 
\eqref{eq:div-tau} with  $\partial_\nu$ instead of $\partial_\tau^\alpha$. We then estimate 
$\partial_\nu\big(\partial_\nu\partial_t^k(\hat\chi E(t))\big)_\nu$ for $k\in\{0,1\}$
as in inequalities \eqref{est:e-nu2-tau} and \eqref{est:De-nu2-tau}. 
Here and in \eqref{est:int}, \eqref{est:e-tau-H1} and  \eqref{est:De-tau-H1} 
we take  a small $\theta>0$ to absorb the $\cH^2$-norms of $\partial_t^k E$ on the right-hand side. 
Using also \eqref{est:h-H2} for the $H$ field, we derive the desired bound in $\cH^2$.

\begin{lemma}[$\cH^2$-estimate]\label{l:H2}
Let $k\in \{0,1\}$ and the assumptions of Theorem~\ref{thm:main} with the exception of the simple connectedness 
of $\Omega$ be satisfied.  Then we can estimate
    \begin{align} 
        \label{est:H2}
        \big\|\partial_t^k (E(t), H(t))&\big\|_{\cH^2(\Omega)}^2 
        + \int _0^t \big\|\partial_t^k (E(s), H(s))\|_{\cH^2(\Omega)}^2\dd s\\
        &\le c\big(z(0)+e(t)+z^2(t)\big) +  c\int_0^t \big(e(s) + z^{3/2}(s)\big)\dd s,\notag
    \end{align}   
  where the constant $c$ does not depend on time $t$.
\end{lemma}

\smallskip

\emph{5) Estimate in $\cH^3$.}   
Since the reasoning is similar to the one presented above, we will omit unnecessesary details here. 
Let $ k= 0$. 

a) We again begin with the magnetic field $H$. We first look at the tangential derivative
$\partial_\tau^\alpha (\hat\chi E)$ with $|\alpha|=2$, where we proceed as in \eqref{est:h-Dtau-H1}
using $\tcurl$-$\tdiv$ estimates.
For $\xi,\zeta\in \{\nu,\tau^1,\tau^2\}$, differentiating the divergence relation \eqref{eq:div} we obtain
\begin{align}\label{eq:div-h-H2}
    \Div\!\big(\mu(H) \partial_\xi\partial_\zeta (\hat\chi H)\big)
    &= \partial_\xi \partial_\zeta([\Div,\hat\chi]\mu(H) H) - [\partial_\xi\partial_\zeta,\Div] (\mu(H)\hat\chi H) \\
    &\qquad-  \Div(\partial_\zeta \mu(H) \, \partial_\xi(\hat\chi H))
    -  \Div(\partial_\xi \mu(H) \, \partial_\zeta(\hat\chi H))
     - \Div(\partial_\xi \partial_\zeta\mu(H) \, \hat\chi H). \notag
\end{align}
Similary, the magnetic  boundary condition in \eqref{eq:maxwell0} yields
\[
  \tr_n\!\big(\mu(H) \partial_\tau^\alpha (\hat\chi H)\big)
   =  [\tr_n,\partial_\tau^\alpha] (\mu(H)\hat\chi H)    + 
  \tr_n\sum_{0\le\beta<\alpha}\binom{\alpha}{\beta}\partial_\tau^{\alpha-\beta}\mu(H) \,\partial_\tau^{\beta}(\hat\chi H).
\]
Using  \eqref{est:z},  from \eqref{eq:maxwell-tau} and the above formulas we deduce the estimates
\begin{align}
    \|\curl\!\big(\partial_\tau^{\alpha} \hat{\chi}H(t)\big)\|_{L^2(\Omega)} 
     &\leq c \big(\|\partial_t H(t)\|_{\cH^2(\Omega)} + \|(E(t),H(t)) \|_{\cH^2(\Omega)} +z(t)\big),\notag \\
\|\Div\!\big(\mu(H(t)) \partial_\tau^{\alpha} (\hat\chi H(t))\big) \|_{L^2(\Omega) } 
 &\leq c\,( \|H(t)\|_{\cH^2(\Omega)} + z(t)),\label{div-curlbounds1} \\
    \|\tr_n\!\big(\mu(H(t)) \partial_\tau^{\alpha} (\hat\chi H(t))\big)\|_{\cH^{1/2}(\Gamma)} 
    &\leq c\,( \|H(t)\|_{\cH^2(\Omega)} + z(t)).\notag
\end{align}
The second-order bound \eqref{est:H2} and Proposition~\ref{prop:div-curl} thus imply
\begin{align}\label{est:h-tau-tau-H1}
\big\|\partial_{\tau}^{\alpha}(\hat{\chi} H(t))\big\|_{\cH^1(\Omega)}^2 
  + \int_0^t\|\partial_{\tau}^{\alpha}(\hat{\chi} H(s))\|_{\cH^1(\Omega)}^2  \dd s
    \le c\big(z(0)+e(t)+z^2(t)\big) + c\int_0^t \big(e(s) +z^{3/2}(s)\big)\dd s.
\end{align}

To handle the mixed derivative $\partial_\nu\partial_\tau$, we use the first équation in \eqref{eq:maxwell-tau} 
with $|\alpha|=1$ and the $\tcurl$-formula \eqref{eq:curl}.  We can then bound the 
$\cH^1$-norm of $\partial_\nu (\partial_\tau(\hat\chi H(t)))^\tau$ by 
 \[ \|\partial_\tau^2(\hat\chi H(t))\|_{\cH^1(\Omega)} + 
   \max_{j\le 1} \|\partial_t^j (E(t), H(t))\big\|_{\cH^2(\Omega)} +z(t).\]
The normal component is treated as in  \eqref{est:h-Dnu-nu-H1}, based on the divergence
relation \eqref{eq:div-chi-h} with $|\alpha|=1$, $\chi$ replaced by $\hat\chi$, and $\partial_x^\alpha$
by $\partial_\tau$. By means of \eqref{eq:div-nu}
and  \eqref{est:z}, the $\cH^1$-norm of the function $\partial_\nu (\partial_\tau(\hat\chi H(t)))_\nu$
is thus controlled by that of $\partial_\nu (\partial_\tau(\hat\chi H(t)))^\tau$ and 
$\partial_\tau^2(\hat\chi H(t))$ plus lower order terms.
Combing these inequalities with \eqref{est:H2} and \eqref{est:h-tau-tau-H1}, we infer
\begin{align}\label{est:h-nu-tau-H1}
    \begin{split}
        \big\|\partial_\nu \partial_{\tau}(\hat{\chi} H(t))\big\|_{\cH^1(\Omega)}^2 
        &+ \int_0^t \|\partial_\nu \partial_{\tau}(\hat{\chi} H(s))\|_{\cH^1(\Omega)}^2  \dd s \\
        &\le c\big(z(0)+e(t)+z^2(t)\big) + c\int_0^t \big(e(s) +z^{3/2}(s)\big)\dd s.
    \end{split}
\end{align}

In this reasoning we can replace  $\partial_\tau$ by $\partial_\nu$, arriving at
\begin{align} \label{est:h-nu-nu-H1}
   \big\|\partial_\nu^2 (\hat{\chi} H(t))\big\|_{\cH^1}^2 
  + \int_0^t\! \|\partial_\nu^2 (\hat{\chi} H(s))\|_{\cH^1}^2  \dd s
   \le c\big(z(0)+e(t)+z^2(t)\big) + c\!\int_0^t \! \big(e(s) +z^{3/2}(s)\big)\dd s.
 \end{align}
Combined with \eqref{est:int}, the estimates   \eqref{est:h-tau-tau-H1},  \eqref{est:h-nu-tau-H1}
and  \eqref{est:h-nu-nu-H1} lead to
\begin{align} 
    \label{est:h-H3}
    \big\|H(t)\big\|_{\cH^3(\Omega)}^2 
    + \int_0^t \| H(s)\|_{\cH^3(\Omega)}^2\dd s
    \le c\big(z(0)+e(t)+z^2(t)\big) + c\int_0^t \big(e(s) +z^{3/2}(s)\big)\dd s.
\end{align}

b) We finally tackle $E$ in $\cH^3$. The third-order tangential derivatives $\partial_\tau^\alpha(\hat\chi E)$
were already treated in estimate \eqref{est:e-tau} with $k=0$, where the lower order-terms on the right-hand side are now
dominated by \eqref{est:H2} and \eqref{est:h-H3}. Let $|\beta|=2$. The second equation in \eqref{eq:maxwell-tau} with 
$|\alpha|=2$ and the $\tcurl$-formula \eqref{eq:curl}
allow us to bound $\partial_\nu(\partial_\tau^\beta(\hat\chi E))^\tau$ in the same fashion. 
The normal component $\partial_\nu(\partial_\tau^\beta(\hat\chi E))_\nu$ can also be controlled via equations 
\eqref{eq:div-tau0} and \eqref{eq:div-nu-sigma}.  We thus arrive at
\begin{align}   \label{est:E-tau2-H1} 
    \big\|& \partial_\tau^\beta(\hat \chi E(t))\big\|_{\cH^1(\Omega)}^2 
    + \int _0^t \big\|\partial_\tau^\beta(\hat \chi E(s))\big\|_{\cH^1(\Omega)}^2\dd s\\
    &\le c\big(z(0)+e(t)+z^2(t)\big) + \theta \int_0^t \big\| E(s)\big\|_{\cH^3(\Omega)}^2\dd s
    + c(\theta)\int_0^t \big(e(s) + z^{3/2}(s)\big)\dd s. \notag
\end{align}
 We replace in system \eqref{eq:maxwell-tau} the tangential derivative 
$\partial_\tau^\alpha$ with $\partial_\nu\partial_\tau$. The second equation therein and formula \eqref{eq:curl} 
provide control of the tangential component $\partial_\nu(\partial_\nu\partial_\tau(\hat\chi E))^\tau$ in $L^2$ via 
inequalities \eqref{est:E-tau2-H1} and \eqref{est:H2}. The related normal component can then be handled through  
the formula \eqref{eq:div-nu-sigma} and the divergence identity 
\eqref{eq:div-tau0} with $\partial_\nu\partial_\tau$ instead of $\partial_\tau^\alpha$. 
In this way we  show the estimate
\begin{align*}
    \big\|& \partial_\tau(\hat \chi E(t))\big\|_{\cH^2(\Omega)}^2 
    + \int _0^t \big\|\partial_\tau(\hat \chi E(s))\big\|_{\cH^2(\Omega)}^2\dd s\\
    &\le c\big(z(0)+e(t)+z^2(t)\big) + \theta \int_0^t  \big\| E(s, t)\big\|_{\cH^3(\Omega)}^2\dd s
    + c(\theta)\int_0^t \big(e(s) + z^{3/2}(s)\big)\dd s.\notag
\end{align*}
The remaining $\partial_\nu^ 3 (\hat \chi E)$-term is managed analogously, resulting in the inequality
\begin{align*}
    \big\|&\hat \chi E(t)\big\|_{\cH^3(\Omega)}^2 
    + \int _0^t \big\|\hat \chi E(s)\big\|_{\cH^3(\Omega)}^2\dd s\\
    &\le c\big(z(0)+e(t)+z^2(t)\big) + \theta \int_0^t \big\| E(s)\big\|_{\cH^3(\Omega)}^2\dd s
    + c(\theta)\int_0^t \big(e(s) + z^{3/2}(s)\big)\dd s.\notag
\end{align*}
 Fixing a sufficiently small number $\theta > 0$, the above inequalities and the interior estimate
 \eqref{est:int} lead to the final bound
 \begin{align}     \label{est:e-H3}
    \big\|E(t)\big\|_{\cH^3(\Omega)}^2 
    + \int _0^t \big\|E(s)\big\|_{\cH^3(\Omega)}^2\dd s
    \le c\big(z(0)+e(t)+z^2(t)\big) + c\int_0^t \big(e(s) + z^{3/2}(s)\big)\dd s.
\end{align}
Equation \eqref{est:h-H3} and \eqref{est:e-H3} now furnish our last result.
\begin{lemma}{[$\cH^3$ estimate]}\label{l:H3}
Let  the assumptions of Theorem~\ref{thm:main} with the exception of the simple connectedness 
of $\Omega$ be satisfied.  Then we can estimate
    \begin{align*} 
        \big\| (E(t), H(t))\big\|_{\cH^3(\Omega)}^2 
        + \int _0^t \!\big\| (E(s), H(s))\big\|_{\cH^3(\Omega)}^3\dd s
        \le c\big(z(0)+e(t)+z^2(t)\big) +  c\!\int_0^t \big(e(s) + z^{3/2}(s)\big)\D s,
    \end{align*}   
  where the constant $c$ does not depend on time $t$.
\end{lemma}
Lemma~\ref{l:H1}, Lemma~\ref{l:H2} and Lemma~\ref{l:H3}  complete the proof of Proposition~\ref{prop:ze}. 
 \end{proof}

\bibliographystyle{plain}
\bibliography{bibliography}


\end{document}